\documentclass[leqno,final]{siamltex}
\usepackage{amsmath}
\usepackage{graphicx,color}
\usepackage{amssymb}

\newtheorem{remark}{Remark}
\newcommand{\argmin}{\operatornamewithlimits{argmin}}
\newcommand{\eps}{\epsilon}

\def\O{\Omega}
\def\D{\mathcal{D}}
\def\F{\mathcal{F}}
\def\P{\mathbb{P}}
\def\E{\mathbb{E}}
\def\R{\mathbb{R}}
\def\({\left(}
\def\){\right)}
\def\div{\mbox{div}}

\begin{document}

\title{Finite Element Methods for the Stochastic Allen-Cahn Equation with Gradient-type 
Multiplicative Noises}
\markboth{XIAOBING FENG AND YUKUN LI AND Yi Zhang}{FEMs for a Stochastic Allen-Cahn Equation}

\author{
Xiaobing Feng\thanks{Department of Mathematics, The University of Tennessee,
Knoxville, TN 37996, U.S.A. ({\tt xfeng@math.utk.edu.}) The work of this author was partially supported 
by the NSF grant DMS-1318486.}
\and
Yukun Li\thanks{Department of Mathematics, The University of Tennessee,
Knoxville, TN 37996, U.S.A. ({\tt yli@math.utk.edu.}) }
\and
Yi Zhang\thanks{Department of Mathematics, The University of Tennessee, 
Knoxville, TN 37996, U.S.A.  ({\tt yzhan112@utk.edu.}) }
}

\maketitle

\begin{abstract}
This paper studies finite element approximations of the stochastic Allen-Cahn equation with 
gradient-type multiplicative noises that are white in time and correlated in space.  
The sharp interface limit as the parameter $\epsilon \rightarrow 0$ of the stochastic 
equation formally approximates a stochastic mean curvature flow which is described by 
a stochastically perturbed geometric law of the deterministic mean 
curvature flow. Both the stochastic Allen-Cahn equation and the stochastic mean curvature 
flow arise from materials science, fluid mechanics and cell biology applications.
Two fully discrete finite element methods which are based on different time-stepping 
strategies for the nonlinear term are proposed. Strong convergence with sharp rates for both fully 
discrete finite element methods is proved. This is done with a crucial help of the H\"{o}lder continuity
in time with respect to the spatial $L^2$-norm and $H^1$-seminorm for the strong solution 
of the stochastic Allen-Cahn equation, which are
key technical lemmas proved in paper. It also relies on the fact that high moments of the 
strong solution are bounded in various spatial and temporal norms.
Numerical experiments are provided to gauge the performance of the proposed fully discrete 
finite element methods and to study the interplay of the geometric evolution and gradient-type noises.
\end{abstract}



\section{Introduction}
The Allen-Cahn equation refers to the following singularly perturbed heat equation:
\begin{align}\label{dac}
u_t - \Delta u + \frac{1}{\epsilon^2} f(u) = 0 \quad \text{in} \ \D_T : = \D \times (0,T),
\end{align}
where $\D \subseteq \R^d$ ($d = 2,3$) is a bound domain, $T > 0$ is a fixed constant. 
Equation \eqref{dac} was proposed by Samuel M. Allen and John. W. Cahn in \cite{AC1979} as a 
mathematical model to describe the phase separation process of a binary alloy quenched at a 
fixed temperature. In the equation, $u$ represents the concentration of one species of
the binary alloy mixture. The positive small parameter $\epsilon$ is known as the interaction 
length of the transition region between the domains $\{ x \in \D: u(x,t) \approx -1 \}$ 
and $\{ x \in \D: u(x,t) \approx 1 \}$, and $f$ is the derivative 
of a smooth double equal well density function $F$ taking its global 
minimum $0$ at $u = \pm 1$. A typical example is
\begin{align}\label{fF}
f(u) = u^3 - u \qquad \text{and} \qquad F(u) = \frac{1}{4} (u^2 - 1)^2,
\end{align}
which will be used in this paper. We note that $t$ in the equation is a fast time, 
representing $\frac{t}{\epsilon^2}$ in the original Allen-Cahn equation. 

To uniquely determine a solution of the Allen-Cahn equation, a boundary and a initial condition
must be prescribed. In this paper we shall consider the following set of the boundary and 
initial conditions:
\begin{alignat}{2} \label{dac-b}
\frac{\partial u}{\partial n} &= 0 &&\quad \text{in} \ \partial \D_T:= \partial \D \times (0,T), \\
\label{dac-i}
u &= u_0 &&\quad  \text{in} \ \D \times \{ 0 \},
\end{alignat}
where $n$ denotes the unit outward normal to $\partial \D$ and $\frac{\partial u}{\partial n}
=\nabla u\cdot n$ denotes the normal derivative of $u$. 

Besides the important role it plays in materials science, the Allen-Cahn equation has also be well 
known because of its connection to a well-known geometric moving interface problem called the mean 
curvature flow, which is governed by the following geometric law:
\begin{align}
\label{dmcf}
V_n(x,t) = -H(x,t) \qquad x \in \Gamma_t,
\end{align}
where $V_n(\cdot,t)$ and $H(\cdot,t)$ respectively stand for the normal velocity and the 
mean curvature of $\Gamma_t$ at the space-time point $(x,t)$. It was proved in \cite{MS1990, ESS1992, I1993} 
that when $\epsilon \to 0$, the zero-level set $\Gamma^{\epsilon}_t := \{ x \in \D: u(x,t) = 0 \}$ 
of the solution of problem \eqref{dac}--\eqref{dac-i} converges to the mean curvature 
flow $\{\Gamma_t\}_{t\geq 0}$. Furthermore, because of the above property the Allen-Cahn equation 
has become a fundamental equation and a building block in the phase field methodology for 
general moving interface problems \cite{AMW98,FHL06}. We refer to \cite{FP2003,NV1997}
and the references therein for an account on numerical methods for problem 
\eqref{dac}--\eqref{dac-i}. 

For application problems of the mean curvature flow, there may exist uncertainty which arises 
from various sources such as thermal fluctuation, impurities of the materials and the 
intrinsic instabilities of the deterministic evolutions. Therefore, it is interesting and necessary 
to consider the stochastic effects, and to study the impact of the noises on regularities 
of the solutions and their longtime behaviors. This leads to considering the following 
stochastically perturbed mean curvature flow \cite{KO1982, RW2013, Y1998}:
\begin{align}
\label{smcf}
V_n(x,t) = -H(x,t) + \delta \mathop{\mathbb{X}}^{\circ}(x, t) \cdot n \qquad x \in \Gamma_t,
\end{align}
where $\mathbb{X}:\mathbb{R}^d\times [0,T]\to \mathbb{R}^d$ and $\stackrel{\circ}{\mathbb{X}}$ 
denotes the Stratonovich derivative of $\mathbb{X}$, $\delta > 0$ is the noise intensity 
that controls the strength of the noise. Formally, the phase field formulation of \eqref{smcf} is 
given by the following stochastic Allen-Cahn equation \cite{RW2013}:
\begin{align}
\label{sac_general}
\frac{d u}{d t}  = \Delta u - \frac{1}{\epsilon^2} f(u)   
+ \delta \nabla u \cdot \mathop{\mathbb{X}}^\circ,
\end{align}
More specifically, we assume $\mathbb{X}$ is a vector field-valued Brownian motion that is white 
in time and smooth in space, i.e., 
\begin{align}
\mathbb{X} (x,t) = X(x) W(t),
\end{align}
where $X: \mathbb{R}^d \longrightarrow \mathbb{R}^d$ is a time independent deterministic 
smooth vector field with compact support in $\D$ and $W(t)$ is a standard $\R$-valued Brownian 
motion (i.e., Wiener process) on a given filtered probability space $(\O,\F,\{ \F_t: t \geq 0 \},\P)$. 
Thus the Stratonovich stochastic partial differential equation (SPDE) \eqref{sac_general} 
becomes
\begin{align}
\label{sac_s}
d u = \Bigl[ \Delta u - \frac{1}{\epsilon^2} f(u) \Bigr] \, dt + \delta \nabla u \cdot X \circ d W(t),
\end{align}
and the corresponding It\^{o} SPDE is given by
\begin{align}
\label{sac}
d u &= \Bigl[ \Delta u - \frac{1}{\epsilon^2} f(u) 
+ \frac{\delta^2}{2} \nabla (\nabla u \cdot X) \cdot X \Bigr] \, dt + \delta \nabla u \cdot X \, d W(t) \\
&=\Bigl[ \Delta u - \frac{1}{\epsilon^2} f(u) 
+\frac{\delta^2}{2} (B : D^2 u + b \cdot \nabla u) \Bigr] \, dt + \delta \nabla u \cdot X \, d W(t) \notag \\
&=\Bigl[ \Delta u - \frac{1}{\epsilon^2} f(u) + \frac{\delta^2}{2} \big(\mbox{div}(B \nabla u) 
+ (b - \text{div}B) \cdot \nabla u\big) \Bigr] \, dt \notag \\
&\qquad + \delta \nabla u \cdot X \, d W(t), \notag
\end{align}
where $B = X \otimes X \in  \R^{d \times d}$ with $B_{ij} = X_i X_j$, $b = (\nabla X) X \in \R^d$ 
with $b_j = \nabla X_j\cdot X$ and $\text{div} B - b = (\text{div} X) X$. Note that the It\^{o} 
SPDE \eqref{sac} has two extra terms which are absorbed in the Stratonovich SPDE \eqref{sac_s}.
The stochastic Allen-Cahn equation \eqref{sac} will be complemented by the boundary and initial 
conditions \eqref{dac-b} and \eqref{dac-i}.

\begin{remark}
We can also consider the following more general vector field \cite{RW2013}
\begin{align*}
\mathbb{X}(x,t) = \sum_{k=1}^m X^k(x) W_k(t),
\end{align*}
where $m$ is a positive integer, $W_k(t)$ are independent $\R$-valued standard Brownian motions,
and $X^k:\mathbb{R}^d \longrightarrow \mathbb{R}^d$ is a time independent deterministic
smooth vector field with compact support in $\D$ for $k=1,2,\cdots, m$.
In this case, the corresponding $B$ and $b$ in the SPDE \eqref{sac} are given by
\begin{align*}
B = \sum_{k=1}^m X^k \otimes X^k \qquad \text{and} \qquad b = \sum_{k=1}^m (\nabla X^k) X^k.
\end{align*}
We remark that the results of this paper can be easily extended to the case with the above
general $\mathbb{X}$.
\end{remark}

In this paper we assume that
\begin{align}
\label{assump-i1}
u_0 &\in C^{\infty}(\bar{\D}), \\
\label{assump-i3}
X &\in C_0^{4,\beta_0}(\D) 
\end{align}
for some $\beta_0 \in (0,1]$. It was proved in \cite[Theorem 4.1]{RW2013} that 
if the domain $\D$ is smooth and the assumptions \eqref{assump-i1}--\eqref{assump-i3} hold, 
there exists a unique strong solution $u$ such that
\begin{align}\label{s_solu-1}
u(x,t) &= u_0(x) + \int_0^t \Bigl[ \Delta u(x,s) - \frac{1}{\epsilon^2} f(u(x,s)) 
+ \frac{\delta^2}{2} \Bigl(B: D^2 u(x,s) \\
&\quad + b \cdot \nabla u(x,s) \Bigr) \Bigr]  \, ds 
+ \int_0^t \delta \nabla u(x,s) \cdot X(x) \, d W(s), \notag \\
\label{s_solu-2}
\frac{\partial u}{\partial n} &= 0 \qquad \text{on} \ \partial \D_T
\end{align}
hold $\P$-almost surely. Moreover, $u(\cdot,t)$ is a continuous $C^{3,\beta}(\bar{\D})$-semimartingale 
for any $0< \beta < \beta_0$. Furthermore, for any multi-index $|\sigma| \leq 3$ and $p \geq 1$, 
there exists a positive constant $C_0 = C (p, \delta,\epsilon)$ such that  
the $p$-th moments of the spatial derivatives of $u$ satisfy
\begin{align}\label{p-bound}
\sup_{t \in [0,T]} \E \Bigl[ \sup_{x \in \D} |\partial^{\sigma} u(x,t) |^p \Bigr] < C_0,
\end{align}
where and throughout this paper $\E[\cdot]$ denotes the expectation operator,

We note that other solution concepts for the SPED problem \eqref{sac}, 
\eqref{dac-b}--\eqref{dac-i}  were also studied in the literature, this 
includes mild solutions based on the semigroup approach \cite{DPZ1992} 
and variational solutions \cite{KR2007} based on the variational approach. 
In this paper we shall assume that there exists a unique strong solution to problem \eqref{sac}, 
\eqref{dac-b}--\eqref{dac-i} satisfying \eqref{s_solu-1}--\eqref{p-bound}. It is clear that the strong solution 
$u(\cdot,t)$ is also an adapted $H^1(\D)$-valued process such that for any $t \in (0,T]$ there holds 
$\P$-almost surely 
\begin{align}
\label{var_solu}
(u(t),\phi) &= (u_0,\phi) - \int_0^t \Bigl( \big(I + \frac{\delta^2}{2} B \big) 
\nabla u(s), \nabla \phi \Bigr) \, ds \\
   &\quad - \frac{1}{\epsilon^2} \int_0^t  \big( f(u(s)),\phi \big) \, ds  
- \frac{\delta^2}{2} \int_0^t \big( (\text{div}B - b)\cdot \nabla u(s), \phi \big)  \, ds \notag \\
   &\quad + \delta \int_0^t (\nabla u(s) \cdot X,\phi) \, d W(s) \qquad \forall \, \phi \in H^1(\D), \notag
\end{align}
where $(\cdot,\cdot)$ denotes the standard inner product on $L^2(\D)$. 

In addition to the results about the well-posedness of strong solutions and the regularity estimates 
given by \eqref{p-bound},  the tightness of the solution in the sharp interface limit as $\epsilon\to 0$  
was also established in \cite{RW2013} for the stochastic Allen-Cahn problem \eqref{sac}, 
\eqref{dac-b}--\eqref{dac-i}. 
However, the rigorous justification of the convergence of the zero level set of the solution $u$  
to the stochastic mean curvature flow \eqref{smcf} is still missing. In \cite{BDP1995, F1995}, 
the stochastic Allen-Cahn equation in one dimension with additive space-time white noise was studied 
and the sharp interface limit was established. Finite element methods for this one dimensional model 
was also proposed in \cite{KKL2007}. But in higher dimensions, it requires spatial correlations 
for the noise otherwise the space-time white noise is too rough to allow the existence of solutions 
in reasonable function spaces. In \cite{FLP2014}, finite element methods for the stochastic mean 
curvature flow of planar curves of graphs were investigated, where the SPDE is quasilinear and arises 
from the level set formulation of the mean curvature flow. However, it is not clear if the results of
\cite{FLP2014} can be extended to higher dimensions. Very recently, finite element methods for a 
stochastic Allen-Cahn equation with Lipschitz continuous multiplicative noises were proposed 
in \cite{P2015}, the strong convergence with suboptimal rates was proved. 

The primary goal of this paper is to develop and analyze two fully discrete finite element 
methods for approximating the strong solution of the stochastic Allen-Cahn problem \eqref{sac},
\eqref{dac-b}--\eqref{dac-i} and to establish strong convergence with optimal 
rates for the proposed fully discrete finite element methods. Our strong convergence 
result will be established under the assumption that the strong solution $u(\cdot,t) \in W^{s,\infty}(\D)$ 
($s \geq 3$) for any $t \in [0,T]$ and it satisfies the following estimates:
\begin{align}
\label{A:Reg}
\sup_{t \in [0,T]} \E \Bigl[ \|u\|^p_{W^{s,\infty}(\D)} \Bigr] \leq C_0 = C(p,\delta,\epsilon) 
\qquad\forall p\geq 1.
\end{align}
We note that the above assumption is reasonable in view of \eqref{p-bound}.
Since we are interested in the case where $\epsilon$ is small, we shall make an effort to track 
the dependence of all constants on the parameter $\epsilon$ in our convergence analysis. 
We remark that for the deterministic Allen-Cahn equation, using a nonstandard technique it was 
proved that the error estimates depend on $\frac{1}{\epsilon}$ only in 
some lower polynomial order \cite{FP2003, FL2014}. Such error estimates can be used to 
prove convergence of the numerical interface to the mean curvature flow. However, it is not clear 
whether the techniques used in \cite{NV1997,FP2003, FL2014} can be extended to the stochastic 
case. As a result, our error estimates are shown to depend on $\frac{1}{\epsilon}$ exponentially. 
Another goal of this paper is to use the stochastic Allen-Cahn equation as a testbed to
develop numerical analysis techniques and machineries which hopefully are applicable 
to other stochastic PDEs. To the best of our knowledge, there is no numerical analysis
result for nonlinear stochastic second order PDEs with gradient-type multiplicative noises in the literature,
this is partially because the gradient-type noises in a SPDE like \eqref{sac} give new challenges 
for error estimates. More specifically, such a multiplicative noise introduces additional 
diffusion and convection terms to the PDE, and the resulting SPDE could be convection-dominated
if the multiplicative noise is too strong. To avoid the latter situation, we shall assume that 
$\delta$ is not too large in \eqref{sac}.

The remaining of the paper is organized as follows. We present in Section~\ref{sec:pre} some 
technical lemmas which play a crucial role in our convergence analysis. In Section~\ref{sec:fem}, 
we propose two fully discrete finite element methods for the stochastic Allen-Cahn problem \eqref{sac},
\eqref{dac-b}--\eqref{dac-i}. These two methods differ only on how the nonlinear term is discretized 
in time, which leads to different conditions on the time step $\tau$ for the well-posedness and 
error estimates. We establish the strong convergence with optimal rates for both fully discrete
finite element methods.  Finally, in Section~\ref{sec:numer} we present some numerical 
experiments to illustrate the performance of the proposed numerical methods and to verify 
the theoretical error estimates obtained in Section~\ref{sec:fem} as well as to examine the 
dynamics of the numerical interfaces, in particular, pertaining to the noise strength.

\section{Preliminary results}\label{sec:pre}
Standard space notation will be used in this paper \cite{BS2008}, in particular, we point out that
$H^k(\D)$ for $k\geq 0$ denotes Sobolev space of order $k$ and 
$(\cdot, \cdot)$ denotes the standard inner product of $L^2(\D)$. Throughout the paper $C$ will 
be used to denote a generic positive constant independent of the parameters $\epsilon$, $\delta$, 
and the space and time mesh sizes $h$ and $\tau$, which can take different values at different occurrences. 

The goal of this section is to establish several technical lemmas that will be crucially
used in Section~\ref{sec:fem}. Each of these lemmas shows a certain property of
the strong solution $u$ of the stochastic Allen-Cahn problem \eqref{sac},
\eqref{dac-b}--\eqref{dac-i}, such as $(\epsilon, \delta)$-explicit energy estimates
and H\"older continuity in time with respect to the $L^2(\D)$- and $H^1(\D)$-norm.

We begin with the following uniform estimate for the expectation of the $p$-th moment of the
Cahn-Hilliard energy functional
\begin{equation*}
J(v) := \int_{\D} \Bigl( \frac{1}{2} |\nabla v|^2 + \frac{1}{\epsilon^2} F(v) \Bigr) \, dx.
\end{equation*}
A similar result can be found in \cite{RW2013} without tracking the dependence on 
parameters in the estimate. We include a proof here for completeness.

\begin{lemma}\label{lem:energy}
Let $u$ be the strong solution to problem \eqref{sac}, \eqref{dac-b}--\eqref{dac-i}. 
There holds for any $p \geq 1$ 
\begin{equation}
\label{energy}
\sup_{t \in [0,T]} \E \bigl[ J(u(t))^p \bigr] 
+ \E \left[ \int_0^t p J(u(s))^{p-1} \bigl\| w(s) \bigr\|^2_{L^2(\D)} \, ds \right] \leq C_1,
\end{equation}
where 
\[
w:= - \Delta u + \frac{1}{\epsilon^2} F'(u)\qquad
\mbox{and}\qquad  
C_1 = e^{C \delta^2 p^2 \|X\|_{C^2(\bar{\D})} T} \bigl[ J(u_0) \bigr]^p.
\]
\end{lemma}

\begin{proof}
By direct calculations we have
\begin{align*}
D (J(u)^p) &= p J(u)^{p-1} \Bigl( -\Delta u + \frac{1}{\epsilon^2} F'(u) \Bigr), \\
D^2 (J(u)^p) &= p (p-1) J(u)^{p-2} \Bigl( -\Delta u 
+ \frac{1}{\epsilon^2} F'(u) \Bigr) \otimes \Bigl( -\Delta u + \frac{1}{\epsilon^2} F'(u) \Bigr) \\
  &\quad + p J(u)^{p-1}  \Bigl( -\Delta + \frac{1}{\epsilon^2} F''(u) \Bigr).
\end{align*}
Now applying It\^{o}'s formula to the functional $\Phi(u(\cdot)) :=J(u(\cdot))^p$ 
and integration by parts, we obtain 
\begin{align}
\label{e:1}
&J(u(t))^p = J(u_0)^p - \int_0^t p J(u(s))^{p-1} \| w(s) \|^2_{L^2(\D)} \, ds \\
&\hskip.05in 
+ \int_0^t p J(u(s))^{p-1} \Bigl( w(s), \frac{\delta^2}{2} \big( B:D^2 u(s) 
+ b \cdot \nabla u(s) \big) \Bigr) \, ds + M_t \notag \\
&\hskip.05in 
+ \frac{\delta^2}{2} \int_0^t p(p-1) J(u(s))^{p-2}\bigl( w(s),\nabla u(s)\cdot X\bigr)^2 \, ds \notag \\
&\hskip.05in
+ \frac{\delta^2}{2} \int_0^t p J(u(s))^{p-1} \int_{\D} \Bigl[ |\nabla(\nabla u(s) \cdot X)|^2 
+ \frac{1}{\epsilon^2} F''(u(s)) (\nabla u(s) \cdot X)^2 \Bigr] \, dx \, ds, \notag
\end{align}
where $M_t$ is the martingale given by
\begin{align}
\label{e:3}
M_t = \delta \int_0^t p J(u(s))^{p-1} \bigl(w(s), \nabla u(s) \cdot X \bigr) \, d W(s).
\end{align}

By integration by parts and a direct calculation, we have
\begin{align}
\label{e:3_2}
\frac{\delta^2}{2} &\int_{\D} \Bigl[ |\nabla(\nabla u(s) \cdot X)|^2 
+ \frac{1}{\epsilon^2} F''(u(s)) (\nabla u(s) \cdot X)^2 \Bigr] \, dx \\
&\qquad \qquad + \Bigl( w(s), \frac{\delta^2}{2} \big( B:D^2 u(s) 
+ b \cdot \nabla u(s) \big) \Bigr) \notag \\
& = \frac{\delta^2}{2} \int_0^t \int_{\D} \Bigl( G(x) : (\nabla u(s) \otimes \nabla u(s)) 
+ \frac{1}{\epsilon^2} g(x) F(u(s)) \Bigr) \, dx \, ds, \notag
\end{align}
where $G(\cdot): \D \longrightarrow \R^{d \times d}$ and $g(\cdot): \D \longrightarrow \R$ are defined by
\begin{align}
\label{e:4}
G_{ij} &= [\partial_k (X_k \partial_l X_l)] \delta_{ij} + \partial_i[X_k \partial_k X_j] 
- 2 \partial_k [X_k \partial_i X_j] + (\partial_k X_i) (\partial_k X_j), \\
g &= \partial_k [X_k \partial_l X_l].
\end{align}

Taking expectation on both sides of \eqref{e:1}, by \eqref{e:3} and the fact that $\E [M_t] = 0$, 
we get the following equation
\begin{align}
\label{e:5}
&\E \left[ J(u(t))^p \right] = J(u_0)^p - \E \left[ \int_0^t p J(u(s))^{p-1} 
\left\| w(s) \right\|^2_{L^2(\D)} \, ds \right] \\
&\hskip.8in 
+ \frac{\delta^2}{2} \E \Bigg[ \int_0^t p J(u(s))^{p-1} 
\int_{\D} \Big( G(x) : (\nabla u(s) \otimes \nabla u(s)) \notag \\
&\hskip1.2in 
+ \frac{1}{\epsilon^2} g(x) F(u(s)) \Big) \, dx \, ds \Bigg] \notag \\
&\hskip.8in 
+ \frac{\delta^2}{2} \E \left[ \int_0^t p(p-1) J(u(s))^{p-2} \left( w(s), 
\nabla u(s) \cdot X \right)^2 \, ds \right]. \notag
\end{align}

Now it remains to estimate the third and fourth terms on the right-hand side of \eqref{e:5}. 
On noting that
\begin{align*}
\|G\|_{C(\bar{\D})} + \|g\|_{C(\bar{\D})} \leq C \|X\|^2_{C^2(\bar{\D})},
\end{align*}
the third term can be bounded as follows
\begin{align}
\label{e:6}
\frac{\delta^2}{2} \E &\left[ \int_0^t p J(u(s))^{p-1} \int_{\D} 
\Bigl( G(x) : (\nabla u(s) \otimes \nabla u(s)) + \frac{1}{\epsilon^2} g(x) F(u(s))\Bigr)\, dx\, ds \right]\\
   &\leq C p \delta^2 \|X\|^2_{C^2(\bar{\D})} \int_0^t \E \left[ J(u(s))^p \right] \, ds. \notag
\end{align}
For the fourth term, by integration by parts and the fact that
\begin{align*}
\Bigl(-\Delta u + \frac{1}{\epsilon^2} F'(u) \Bigr) \nabla u 
= - \nabla \cdot (\nabla u \otimes \nabla u) + \nabla \Bigl( \frac{1}{2} |\nabla u|^2 
+ \frac{1}{\epsilon^2} F(u) \Bigr),
\end{align*}
we have
\begin{align}
\label{e:7}
\frac{\delta^2}{2} \E &\left[ \int_0^t p(p-1) J(u(s))^{p-2}\left(w(s),\nabla u(s)\cdot X\right)^2\,ds \right] \\
&\qquad
\leq C p (p-1) \delta^2 \|X\|^2_{C^1(\bar{\D})} \int_0^t \E \left[ J(u(s))^p \right] \, ds. \notag
\end{align}

Finally, combine \eqref{e:5}--\eqref{e:7} and apply the Gronwall's inequality, we obtain 
the desired estimate \eqref{energy}.
\end{proof}

\begin{remark}
In the case of $p>1$ in Lemma~\ref{lem:energy}, the fourth term on the right-hand 
side of \eqref{e:5} can also be bounded by
\begin{align}
\label{e:8}
\frac{\delta^2}{2}\E &\left[\int_0^t p(p-1) J(u(s))^{p-2}\left(w(s),\nabla u(s)\cdot X\right)^2\, ds \right] \\
&\leq (p-1) \delta^2 \|X\|^2_{C(\bar{\D})} \E \left[ \int_0^t p J(u(s))^{p-1} 
\left\| w(s) \right\|^2_{L^2(\D)} \, ds \right]. \notag
\end{align}
Therefore \eqref{energy} is replaced by
\begin{align}
\label{energy2}
\sup_{t \in [0,T]} \E \left[ J(u(t))^p \right] 
+ C_p \E \left[ \int_0^t p J(u(s))^{p-1} \left\| w(s) \right\|^2_{L^2(\D)} \, ds \right] \leq \tilde{C}_1,
\end{align}
where 
\[
C_p = 1 - (p-1) \delta^2 \|X\|^2_{C(\bar{\D})} \qquad\mbox{and}\qquad
\tilde{C}_1 = e^{C \delta^2 p \|X\|_{C^2(\bar{\D})} T} \left[ J(u_0) \right]^p.
\]
However, \eqref{energy2} implicitly imposes a restriction on $\delta$ with respect to $p$, 
namely,  $\delta \leq 1/(\sqrt{p-1} \|X\|_{C(\bar{\D})})$.
\end{remark}

\begin{remark}
Note that in the case of $p=1$ and $\delta = 0$, we recover the the following deterministic energy law:
\begin{align*}
\sup_{t \in [0,T]} J(u(t)) + \int_0^t \left\| w(s) \right\|_{L^2(\D)} \, ds \leq J(u_0).
\end{align*}
\end{remark}

Next we derive H\"{o}lder continuity in time for the solution $u$ with respect to the spatial $L^2$-norm 
and the $H^1$-seminorm.  Both results play a key role in our error analysis in Section~\ref{sec:fem}. 
In a sense they substitute the time derivative of the solution, which does not exist in the stochastic
case, appeared in the deterministic error analysis. We also note that, as expected, the constants in 
these estimates depend on $\epsilon^{-1}$ in some polynomial orders.

\begin{lemma}\label{lem:e2}
Let $u$ be the strong solution to problem \eqref{sac}, \eqref{dac-b}--\eqref{dac-i}. 
For any $s, t \in [0,T]$ with $s < t$ and $p \geq 2$, we have
\begin{align} \label{e2}
&\frac{p}{2}\, \E \left[ \int_{s}^t \| u(\zeta) - u(s) \|^{p-2}_{L^2(\D)} 
\|\nabla(u(\zeta)-u(s)) \|^2_{L^2(\D)}\, d\zeta \right] \\
&\qquad + \frac{\delta^2 p}{4}\, \E \left[ \int_{s}^t \| u(\zeta) - u(s) \|^{p-2}_{L^2(\D)} 
\|\nabla(u(\zeta)-u(s)) \cdot X \|^2_{L^2(\D)} \, d\zeta \right]  \notag \\
&\qquad +\E \Bigl[ \|u(t)-u(s) \|_{L^2(\D)}^p \Bigr]  \leq C_2 (t-s) \notag
\end{align}
for some $C_2 = C(p, \delta, \|X||_{C^2(\bar{\D})}, T) \epsilon^{-p/2} \bigl( [J(u_0)]^{p/2} + C \bigr)>0$.
\end{lemma}

\begin{proof}
Applying It\^{o}'s formula to the functional 
$\Phi(u(\cdot)) := \|u(\cdot)-u(s)\|_{L^2(\D)}^p$ with fixed $s \in [0,T)$ 
and using integration by parts, we obtain
\begin{align}
\label{e2:1}
&\|u(t)-u(s)\|_{L^2(\D)}^p \\
&\quad = p \int_{s}^t \| u(\zeta) - u(s) \|^{p-2}_{L^2(\D)} 
\Bigl( u(\zeta)-u(s), \Delta u(\zeta) - \frac{1}{\epsilon^2} f(u(\zeta)) \Bigr) \, d\zeta \notag \\
&\qquad  + \frac{\delta^2 p}{2} \int_{s}^t \| u(\zeta) - u(s) \|^{p-2}_{L^2(\D)} 
\Bigl( u(\zeta)-u(s), B:D^2u(\zeta) + b \cdot \nabla u(\zeta) \Bigr) \, d\zeta \notag \\
&\qquad + p \int_{s}^t  \| u(\zeta) - u(s) \|^{p-2}_{L^2(\D)} 
\Bigl( u(\zeta)-u(s), \delta \nabla u(\zeta) \cdot X \Bigr) \,d W(\zeta)  \notag \\
&\qquad + \frac{\delta^2 p}{2} \int_s^t (p-2) \| u(\zeta) - u(s) \|^{p-4}_{L^2(\D)} 
\bigl(u(\zeta) - u(s), \nabla u(\zeta) \cdot X \bigr)^2 \, d \zeta \notag \\
&\qquad + \frac{\delta^2 p}{2} \int_{s}^t \| u(\zeta) - u(s) \|^{p-2}_{L^2(\D)} 
\| \nabla u(\zeta) \cdot X\|^2_{L^2(\D)} \,d\zeta \notag\\
&\quad = I + II + III + IV + V. \notag
\end{align}

We only consider the case $p>2$ since the proof for the case $p=2$ is simpler. 
By Young's inequality, we have
\begin{align}
\label{e2:2}
&\E \left[ p \int_{s}^t \| u(\zeta) - u(s) \|^{p-2}_{L^2(\D)} 
\Bigl( u(\zeta)-u(s), \Delta u(\zeta) \Bigr) \, d\zeta \right] \\
&\quad = - \E \left[ p \int_{s}^t \| u(\zeta) - u(s) \|^{p-2}_{L^2(\D)} 
\| \nabla u(\zeta) - \nabla u(s) \|^2_{L^2(\D)} \, d\zeta \right] \notag \\
&\qquad - \E \left[ p \int_{s}^t \| u(\zeta) - u(s) \|^{p-2}_{L^2(\D)} 
\Bigl( \nabla u(\zeta) - \nabla u(s), \nabla u(s) \Bigr) \, d\zeta \right] \notag \\
&\quad \leq -\frac{p}{2} \E \left[ \int_{s}^t \| u(\zeta) - u(s) \|^{p-2}_{L^2(\D)} 
\| \nabla u(\zeta) - \nabla u(s) \|^2_{L^2(\D)} \, d\zeta \right] \notag \\
&\qquad + \frac{p}{2} \E \left[ \int_{s}^t \| u(\zeta) - u(s) \|^{p-2}_{L^2(\D)} 
\| \nabla u(s) \|^2_{L^2(\D)} \, d\zeta \right] \notag \\
&\quad \leq -\frac{p}{2} \E \left[ \int_{s}^t \| u(\zeta) - u(s) \|^{p-2}_{L^2(\D)} 
\| \nabla u(\zeta) - \nabla u(s) \|^2_{L^2(\D)} \, d\zeta \right]  \notag \\
&\qquad + \frac{p-2}{2} \int_s^t \E \left[ \|u(\zeta)-u(s)\|_{L^2(\D)}^p \right] \, d \zeta 
+ \E \left[ \| \nabla u(s) \|^{p}_{L^2(\D)} \right] (t-s). \notag
\end{align}
Similarly, we have
\begin{align}
\label{e2:3}
&\E \left[ p \int_{s}^t \| u(\zeta) - u(s) \|^{p-2}_{L^2(\D)} 
\Bigl( u(\zeta)-u(s), - \frac{1}{\epsilon^2} f(u(\zeta)) \Bigr) \, d\zeta \right] \\
&\leq \E \bigg[ p \int_{s}^t \| u(\zeta) - u(s) \|^{p-2}_{L^2(\D)} \notag \\
&\qquad \times \Bigl( \frac{1}{2 \epsilon} \| u(\zeta) (u(\zeta)-u(s)) \|^2_{L^2(\D)} 
+ \frac{2}{\epsilon^3} \|F(u(\zeta))\|_{L^1(\D)} \Bigr) \, d\zeta \bigg] \notag \\
&\leq 2 (p-2) \int_s^t \E \Bigl[ \| u(\zeta) - u(s) \|^p_{L^2(\D)} \Bigr] \, d \zeta \notag \\
&+ C_p \epsilon^{-p/2} \Bigl( \sup_{s \leq \zeta \leq t} \E \left[ \|u(\zeta)\|^{2p}_{L^4(\D)} \right] 
+ \sup_{s \leq \zeta \leq t} \E \Bigl[ \Bigl\|\frac{1}{\epsilon^2} F(u(\zeta)) \Bigr\|_{L^1(\D)} ^{p/2}\Bigr] \Bigr)(t-s). 
\notag
\end{align}

Using \eqref{e2:2}--\eqref{e2:3}, the first term on the right-hand side of \eqref{e2:1} can be 
bounded as follows
\begin{align}
\label{e2:4}
&\E [I] \leq -\frac{p}{2} \E \left[ \int_{s}^t \| u(\zeta) - u(s) \|^{p-2}_{L^2(\D)} 
\| \nabla u(\zeta) - \nabla u(s) \|^2_{L^2(\D)} \, d\zeta \right] \\
&\hskip.1in + \frac{5 (p-2)}{2} \int_s^t \E \left[ \| u(\zeta) - u(s) \|^p_{L^2(\D)} \right] \, d \zeta 
+ \E \left[ \| \nabla u(s) \|^{p}_{L^2(\D)} \right] (t-s) \notag \\
&\hskip.1in + C_p \epsilon^{-p/2} \Bigl( \sup_{s \leq \zeta \leq t} 
\E \left[ \|u(\zeta)\|^{2p}_{L^4(\D)} \right] 
+ \sup_{s \leq \zeta \leq t} \E \Bigl[ \Bigl\| \frac{1}{\epsilon^2} F(u(\zeta)) \Bigr\|_{L^1(\D)} ^{p/2} \Bigr] \Bigr) (t-s). \notag
\end{align}

Again, by using Young's inequality and integration by parts the second term on the 
right-hand side of \eqref{e2:1} can be bounded by  
\begin{align}
\label{e2:5}
\E[II] 
&\leq - \frac{\delta^2 p}{4} \E \left[ \int_s^t \| u(\zeta) - u(s) \|^{p-2}_{L^2(\D)} 
\| \nabla u(\zeta) \cdot X - \nabla u(s) \cdot X \|^2_{L^2(\D)} \, d \zeta \right] \\
&\quad + \frac{\delta^2}{2} (p-2) \int_s^t \E \left[ \| u(\zeta) - u(s) \|^p_{L^2(\D)} \right] \, d \zeta \notag \\
&\quad + \frac{\delta^2}{2} (\|X\|^p_{C(\bar{\D})} + \|X\|^{2p}_{C^1(\bar{\D})}) 
\sup_{s \leq \zeta \leq t} \E \left[ \| \nabla u(\zeta)\|^{p}_{L^2(\D)} \right] (t-s). \notag
\end{align}

Next, it follows from Young's inequality that  
\begin{align}
\label{e2:6}
\E \left[ IV + V \right] &\leq \E \left[ \frac{\delta^2}{2} p (p-1) 
\int_s^t \| u(\zeta) - u(s) \|^{p-2}_{L^2(\D)} \| \nabla u(\zeta) \cdot X \|^2_{L^2(\D)} \, d \zeta \right] \\
&\leq \frac{\delta^2}{2} (p-1)(p-2) \int_s^t \E \left[ \| u(\zeta) - u(s) \|^{p}_{L^2(\D)} \right] \, d \zeta \notag \\
&\quad + (p-1) \delta^2 \|X\|^p_{C(\bar{\D})} \sup_{s \leq \zeta \leq t} 
\E \left[ \| \nabla u(\zeta) \|^p_{L^2(\D)} \right] (t-s). \notag
\end{align}

Combining \eqref{e2:1}, \eqref{e2:4}--\eqref{e2:6} and using the facts that
 $\|u\|^4_{L^4(\D)} \leq 8 \epsilon^2 J(u) + C$ and $E[III] = 0$, we have
\begin{align} \label{e2:7}
&\E \left[ \|u(t)-u(s) \|_{L^2(\D)}^p \right] \\
&\qquad + \frac{p}{2} \E \left[ \int_{s}^t \| u(\zeta) - u(s) \|^{p-2}_{L^2(\D)} 
\|\nabla(u(\zeta)-u(s)) \|^2_{L^2(\D)}\, d\zeta \right] \notag \\
&\qquad + \frac{\delta^2 p}{4} \E \left[ \int_{s}^t \| u(\zeta) - u(s) \|^{p-2}_{L^2(\D)} 
\|\nabla(u(\zeta)-u(s)) \cdot X \|^2_{L^2(\D)} \, d\zeta \right] \notag \\
&\quad\leq \frac{p-2}{2} (p \delta^2 + 5) 
\int_s^t \E \left[ \| u(\zeta) - u(s) \|^p_{L^2(\D)} \right] \, d \zeta \notag \\
&\qquad + C(p, \delta^2, \|X\|_{C^1(\bar{\D})}) \epsilon^{-p/2} 
\Bigl( \sup_{s \leq \zeta \leq t} \E \left[ J(u(\zeta)^{p/2}) \right] + C \Bigr) (t-s). \notag
\end{align}

Finally, the desired estimate \eqref{e2} follows from \eqref{e2:7}, the Gronwall's inequality 
and Lemma~\ref{lem:energy}. The proof is complete.
\end{proof}

\begin{lemma}\label{lem:e3}
Let $u$ be the strong solution to problem \eqref{sac}, \eqref{dac-b}--\eqref{dac-i}. 
For any $s,t \in [0,T]$ with $s < t$, we have
\begin{align}
\label{e3}
\E \big[ \|\nabla (u(t)-u(s)) \|_{L^2(\D)}^2 \big] + \frac{1}{2} \E\left[ \int_{s}^t \|\Delta(u(\zeta)-u(s))\|_{L^2(\D)}^2 \, d\zeta \right] \leq C_3 (t-s),
\end{align}
where
\begin{align*}
C_3 &= C \Bigl(\delta^2 \|X\|^2_{C^1(\bar{\D})} + \delta^4 \|X\|^4_{C^1(\bar{\D})} 
+ \frac{1}{\epsilon^2} + 1 \Bigr) \Bigl( C_1 + \sup_{s \leq \zeta \leq t} 
\E \left[ \|\Delta u(\zeta)\|^2_{L^2(\D)} \right] \Bigr).
\end{align*}
\end{lemma}

\begin{proof}
Applying It\^{o}'s formula to the functional $\Phi(u(\cdot)) :=\|\nabla u(\cdot) 
- \nabla u(s)\|_{L^2(\D)}^2$ with fixed $s \in [0,T)$ and using integration by parts, we get
\begin{align}
\label{e3:1}
&\|\nabla u(t)-\nabla u(s)\|_{L^2(\D)}^2 
= -2 \int_{s}^t (\Delta u(\zeta)-\Delta u(s),\Delta u(\zeta) ) \, d\zeta \\
&\quad 
+ 2\int_s^t \Bigl( \Delta u(\zeta) - \Delta u(s), \frac{1}{\epsilon^2} f(u(\zeta)) \Bigr) \, d\zeta \notag \\
&\quad 
- \delta^2 \int_s^t \Bigl( \Delta u(\zeta) - \Delta u(s), B: D^2 u(\zeta) 
+ b \cdot \nabla u(\zeta) \Bigr) \, d\zeta \notag \\
&\quad
-2\delta\int_{s}^t ( \Delta u(\zeta) - \Delta u(s), \nabla u(\zeta)\cdot X ) \, d W(\zeta)
\notag\\
&\quad
+\delta^2 \int_{s}^t \int_{\D} |\nabla(\nabla u(\zeta) \cdot X)|^2 \,dx \, d \zeta. \notag
\end{align}

The first term on the right-hand side of \eqref{e3:1} can be estimated by Schwarz inequality 
as follows
\begin{align} \label{e3:2}
&-2 \E \left[ \int_{s}^t (\Delta u(\zeta)-\Delta u(s),\Delta u(\zeta) ) \, d \zeta \right] \\
&= -2 \E \left[ \int_s^t \|\Delta u(\zeta)-\Delta u(s) \|^2_{L^2(\D)} \, d\zeta 
+ \int_{s}^t (\Delta u(\zeta)-\Delta u(s),\Delta u(s) ) \, d \zeta \right] \notag \\
&\leq - \E \left[ \int_s^t \|\Delta u(\zeta)-\Delta u(s)\|^2_{L^2(\D)} \, d\zeta \right] 
+ \|\Delta u(s)\|^2_{L^2(\D)} (t-s). \notag
\end{align}
For the second term on the right-hand side of \eqref{e3:1}, by Sobolev embedding $H^1(\D) \hookrightarrow L^6(\D)$ for $d \leq 3$, we have
\begin{align} \label{e3:3}
2\E &\left[ \int_s^t \Bigl( \Delta u(\zeta)-\Delta u(s), \frac{1}{\epsilon^2} f(u(\zeta)) \Bigr) \, d\zeta \right] \\
&\quad 
\leq \frac{2}{\epsilon^2} \E \left[ \int_s^t \left( \| \Delta u(\zeta)-\Delta u(s) \|_{L^2(\D)}^2 
+ \| f(u(\zeta)) \|_{L^2(\D)}^2 \right) \, d \zeta \right] \notag \\
&\quad 
\leq \frac{C}{\epsilon^2} \Bigl( \sup_{s \leq \zeta \leq t} 
\E \left[ \|\Delta u(\zeta)\|_{L^2(\D)}^2 \right] 
+ \sup_{s \leq \zeta \leq t} \E \left[ J(u(\zeta))^3 \right] + C \Bigr) (t-s) \notag.
\end{align}

Next we bound the third and fifth terms on the right-hand side of \eqref{e3:1} as follows
\begin{align} \label{e3:4}
&\E \left[ - \delta^2 \int_s^t \Big( \Delta u(\zeta) - \Delta u(s), B: D^2 u(\zeta) 
+ b \cdot \nabla u(\zeta) \Big) \, d\zeta \right] \\
&\quad 
\leq \frac{1}{2} \E \left[ \int_s^t \|\Delta u(\zeta) - \Delta u(s)\|^2_{L^2(\D)}\,d \zeta \right] 
+ \frac{1}{2} \delta^4 \|X\|^4_{C^1(\bar{\D})} \notag \\
&\qquad 
\times \Bigl(\sup_{s \leq \zeta \leq t} \E \left[ \|\Delta u(\zeta)\|^2_{L^2(\D)} \right] 
+ \sup_{s \leq \zeta \leq t} \E \left[ \| \nabla u(\zeta) \|^2_{L^2(\D)} \right] \Bigr) (t-s), \notag
\end{align}
and
\begin{align}
\label{e3:5}
&\E \left[ \delta^2 \int_s^t \int_{\D} |\nabla(\nabla u(\zeta) \cdot X)|^2 \, dx \, d\zeta \right] 
\leq C \delta^2 \|X\|^2_{C^1(\bar{\D})} \\
&\qquad 
\times \Bigl( \sup_{s \leq \zeta \leq t} \E \left[ \|\Delta u(\zeta)\|^2_{L^2(\D)} \right] 
+\sup_{s \leq \zeta \leq t} \E \left[ \| \nabla u(\zeta)\|^2_{L^2(\D)} \right] \Bigr) (t-s). \notag
\end{align}
The desired estimate \eqref{e3} then follows from \eqref{e3:1}--\eqref{e3:5}, 
Lemma~\ref{lem:energy} and the fact that the expectation of the fourth term on the 
right-hand side of \eqref{e3:2} is $0$.
\end{proof}

Finally, we include a H\"{o}lder continuity estimate for the nonlinear increment $f(u(t))-f(u(s)$ 
in the $L^2$-norm which will be useful to control the nonlinear term in the error analysis.

\begin{lemma}\label{lem:e4}
Let $u$ be the strong solution to problem \eqref{sac}, \eqref{dac-b}--\eqref{dac-i}.
For any $s,t \in [0,T]$ with $s < t$, we have
\begin{align} \label{e4}
\E \big[ \|f(u(t))-f(u(s))\|_{L^2(\D)}^2 \big]\leq C_4 (t-s),
\end{align}
where $C_4 = C(\delta, \eps, \|X\|_{C^1(\bar{\D})}, C_0, C_1)$.
\end{lemma}

\begin{proof}
Applying It\^{o}'s formula to the functional $\Phi(u(\cdot)) 
:= \|f(u(\cdot))-f(u(s))\|_{L^2(\D)}^2$ with fixed $s \in [0,T)$, we obtain
\begin{align}\label{e4:1}
&\|f(u(t))-f(u(s))\|_{L^2(\D)}^2 = 2\int_s^t \int_{\D} \big( f(u(\zeta))-f(u(s))\big)f'(u(\zeta)) \\
&\qquad 
\times \Bigl[ \Delta u(\zeta)-\frac{1}{\eps^2}f(u(\zeta))
+ \frac{\delta^2}{2} \div(B \nabla u(\zeta))+\frac{\delta^2}{2}(b-\div B) 
\cdot \nabla u(\zeta) \Bigr] \, dx \, d\zeta \notag\\
&\quad 
+ 2 \delta \int_s^t \int_{\D} \bigl( f(u(\zeta))-f(u(s)) \bigr) f'(u(\zeta)) \nabla u(\zeta) 
\cdot X \, dx \, d W(\zeta) \notag\\
&\quad 
+ \delta^2 \int_s^t \int_{\D} \bigl( f(u(\zeta))-f(u(s)) \bigr) f''(u(\zeta)) 
|\nabla u(\zeta) \cdot X|^2 \, dx \, d\zeta \notag\\
&\quad 
+ \delta^2 \int_s^t \int_{\D}[f'(u(\zeta))]^2|\nabla u(\zeta) \cdot X|^2 \, dx \, d\zeta. \notag
\end{align}

Taking the expectation on both sides, it follows from integration by parts, Young's inequality 
and Sobolev embedding $H^1(\D) \hookrightarrow L^6(\D)$ for $d \leq 3$ that
\begin{align}\label{e4:2}
& \E \left[ \|f(u(t))-f(u(s))\|_{L^2(\D)}^2 \right] 
\leq C (1 + \delta^2 \|X\|_{C^1(\D)}^2) (t-s) \\
&\qquad 
\times \Bigl( C + (1+\epsilon^{-2}) \sup_{s \leq \zeta \leq t} 
\E \left[ \|u(\zeta)\|_{L^{\infty}(\D)}^6 \right] 
+ \sup_{s \leq \zeta \leq t} \E \left[ J(u(\zeta))^3 \right] \Bigr). \notag
\end{align}

Finally, the desired estimate \eqref{e4} follows from \eqref{A:Reg}, \eqref{e4:2} and Lemma~\ref{lem:energy}.
\end{proof}

\section{Finite element methods}\label{sec:fem}
In this section we propose two fully discrete finite element methods for problem 
\eqref{sac}, \eqref{dac-b}--\eqref{dac-i} and establish their strong convergence
with optimal rates.

Let $t_n = n \tau$ ($n = 0, 1, \dots, N$) be a uniform partition of $[0,T]$ with $\tau = T/N$ and $
\mathcal{T}_h$ be a quasi-uniform triangulation of $\D$. We consider the finite element space
\begin{align*}
V^r_h := \{ v_h \in H^1(\D): v_h|_K \in \mathcal{P}_r(K) \quad \forall \, K \in \mathcal{T}_h \},
\end{align*}
where $\mathcal{P}_r(K)$ denotes the space of all polynomials of degree not exceeding a 
given positive integer $r$ on $K \in \mathcal{T}_h$. Our fully discrete finite element 
methods for problem \eqref{sac}, \eqref{dac-b}--\eqref{dac-i} are defined by seeking 
an $\F_{t_n}$-adapted $V^r_h$-valued process $\{ u_h^n \}$ ($n = 0, 1, \dots, N$) such 
that $\P$-almost surely
\begin{align}
\label{dfem}
&(u^{n+1}_h, v_h) + \tau \Bigl( \big(I + \frac{\delta^2}{2} B\big) \nabla u^{n+1}_h, \nabla v_h \Bigr) + \tau \frac{1}{\epsilon^2} (f^{n+1}, v_h) \\
&\qquad 
= (u^n_h, v_h) - \tau\frac{\delta^2}{2} \big( (\mbox{div} B - b) \cdot \nabla u^{n}_h, v_h \big)
+ \delta (\nabla u^n_h \cdot X, v_h) \, \Delta W_{n+1} \notag
\end{align}
for all $v_h \in V^r_h$, where $\Delta W_{n+1} := W(t_{n+1}) - W(t_n) \sim \mathcal{N} (0,\tau)$ and
\begin{align}
\label{fn}
f^{n+1} := (u^{n+1}_h)^3 - u^{n+1}_h \qquad \text{or} \qquad f^{n+1} := (u^{n+1}_h)^3 - u^{n}_h.
\end{align}
We remark that in the deterministic case (i.e., $\delta = 0$), \eqref{fn} corresponds 
to a fully implicit scheme or a convex-splitting scheme. We choose $u_h(0)  = P_h u_0$ 
to complement \eqref{dfem} where $P_h: L^2(\D) \longrightarrow V^r_h$ is the 
$L^2$-projection operator defined by
\begin{align*}
\bigl(P_h w, v_h\bigr) = (w, v_h) \qquad v_h \in V^r_h.
\end{align*}
The following error estimate results are well-known \cite{C1978, BS2008}:
\begin{align}
\label{Ph1}
&\|w - P_h w \|_{L^2(\D)} + h \| \nabla (w - P_h w) \|_{L^2(\D)} 
\leq C h^{\min{(r+1,s)}} \|w\|_{H^s(\D)},\\ 
\label{Ph2}
&\|w - P_h w\|_{L^\infty(\D)} \leq C h^{2-\frac{d}{2}} \|w\|_{H^2(\D)}  
\end{align}
for all $w \in H^s(\D)$.

We consider a convex decomposition $F(v) = F^+(v) - F^-(v)$ where
\begin{align}
\label{F+-}
F^+(v) = \frac{1}{4} (v^4 + 1) \qquad \text{and} \qquad F^-(v) = \frac{1}{2} v^2,
\end{align}
and define functionals
\begin{align}
\label{GH}
&I(v) := \frac{1}{2} (v,v) + \frac{\tau}{2} (\nabla v, \nabla v) 
+ \frac{\tau \delta^2}{2} (\nabla v \cdot X, \nabla v \cdot X) + \frac{\tau}{\epsilon^2} (F(v),1) \\
&\qquad - (u^n_h,v) + \frac{\tau \delta^2}{2} \big( (\text{div} B - b) 
\cdot \nabla u^n_h, v \big) - \delta \big((\nabla u^n_h \cdot X) \Delta W_{n+1}, v \big), \notag \\
&H(v) := \frac{1}{2} (v,v) + \frac{\tau}{2} (\nabla v, \nabla v) 
+ \frac{\tau \delta^2}{2} (\nabla v \cdot X, \nabla v \cdot X) + \frac{\tau}{\epsilon^2} (F^+(v),1) \\
&\qquad - \Bigl( 1 + \frac{\tau}{\epsilon^2} \Bigr) (u^n_h,v) 
+ \frac{\tau \delta^2}{2} \big( (\text{div} B - b) \cdot \nabla u^n_h, v \big) - \delta \big((\nabla u^n_h \cdot X) \Delta W_{n+1}, v\big). \notag
\end{align}
It is clear that $H(v)$ is strictly convex for all $h, \tau >0$ and $I(v)$ is strictly 
convex when $\tau \leq \epsilon^2$. Then a straightforward calculation shows that the discrete 
problem \eqref{dfem} is equivalent to the following finite-dimensional convex minimization problems:
\begin{align}
\label{min}
u^{n+1}_h &= \argmin_{v_h \in V^r_h} I(v_h), \qquad \text{when} 
\quad f^{n+1} = (u^{n+1}_h)^3 - u^{n+1}_h \quad \text{and} \quad \tau \leq \epsilon^2 \\
u^{n+1}_h &= \argmin_{v_h \in V^r_h} H(v_h), \qquad \text{when} 
\quad f^{n+1} = (u^{n+1}_h)^3 - u^{n}_h.
\end{align}
Therefore, the existence and uniqueness of problem \eqref{dfem} is guaranteed for all 
$h, \tau > 0$ when $f^{n+1} = (u^{n+1}_h)^3 - u^{n}_h$ and for all $\tau \leq \epsilon^2$ 
when $f^{n+1} = (u^{n+1}_h)^3 - u^{n+1}_h$.

\begin{remark}
We can also consider a modified scheme
\begin{align}
\label{dfem2}
&(u^{n+1}_h, v_h)  + \tau \Bigl( \bigl(I + \frac{\delta^2}{2} B\bigr) 
\nabla u^{n+1}_h, \nabla v_h \Bigr) + \tau \frac{1}{\epsilon^2} (f^{n+1}, v_h) \\
&\qquad \qquad + \tau \frac{\delta^2}{2} \Bigl( (\mbox{\rm div} B - b) 
\cdot \nabla u^{n+1}_h, v_h \Bigr) \notag \\
&\quad = (u^n_h, v_h) + \delta (\nabla u^n_h \cdot X, v_h) \, \Delta W_{n+1} 
\qquad \forall \, v_h \in V_h, \notag
\end{align}
where we replace the term $\tau \frac{\delta^2}{2} \bigl( (\mbox{\rm div} B - b) 
\cdot \nabla u^{n}_h, v_h \bigr) $ in \eqref{dfem} by $\tau \frac{\delta^2}{2} 
\bigl( (\mbox{\rm div} B - b) \cdot \nabla u^{n+1}_h, v_h \bigr) $ in \eqref{dfem2}. 
Clearly, \eqref{dfem} has a simpler form and the stiffness matrix is symmetric. 
On the other hand, \eqref{dfem2} has one more convection contribution making the 
corresponding stiffness matrix non-symmetric.
\end{remark}

\begin{remark}
Due to the time discretization, it is unclear whether the discrete analog of energy 
bound in Lemma~\ref{lem:energy} is valid. However, the error estimate below does not 
require any discrete solution estimate.
\end{remark}

Let $e^n := u(t_n) - u_h^n$ ($n=0,1,2,...,N$), we now derive error estimates for $e^n$.
\begin{theorem}\label{thm:derrest}
Let $u$ and $\{ u_h^n \}_{n=1}^N$ denote respectively the solutions of problem \eqref{sac}, 
\eqref{dac-b}--\eqref{dac-i} and  scheme \eqref{dfem}, and $u$ satisfies \eqref{A:Reg}. 
In the case of $f^{n+1} = (u_h^{n+1})^3 - u_h^{n+1}$, under the following mesh constraint
\begin{align} \label{meshc1}
\tau \leq C \left( \epsilon^{-2} + \delta^4 \right)^{-1},
\end{align}
there holds
\begin{align} \label{derrest1}
&\sup_{0 \leq n \leq N} \E \left[ \| e^n \|^2_{L^2(\D)} \right] 
+ \E \left[ \sum_{n=1}^N \tau \|\nabla e^n\|^2_{L^2(\D)}  \right] \\
& \leq T \bigg[ \Bigl( C_3(1+\delta^2+\delta^4) + \frac{C_4}{\epsilon^2} \Bigr) \tau \notag \\
   &\qquad + C_0 (1 + \delta^2 + \delta^4) h^{2 \min{(r,s-1)}} 
+ \Bigl(C_0 + \frac{C_0^2}{\epsilon^2} \Bigr) h^{2 \min{(r+1,s)}} \bigg] 
e^{C (\delta^4 + \frac{1}{\epsilon^2}) T} \notag \\
   & \leq C(\delta, \epsilon, T, \|X\|, C_0, C_2, C_3, C_4) \bigl[\tau + h^{2 \min{(r,s-1)}}\bigr]. 
\notag
\end{align}
In the case of $f^{n+1} = (u_h^{n+1})^3 - u_h^{n}$, under the following weaker mesh constraint
\begin{align}
\label{meshc2}
\tau \leq C (1 + \delta^4)^{-1},
\end{align}
we have
\begin{align} \label{derrest2}
&\sup_{0 \leq n \leq N} \E \left[ \| e^n \|^2_{L^2(\D)} \right]
 + \E \left[ \sum_{n=1}^N \tau \|\nabla e^n\|^2_{L^2(\D)}  \right] \\
&\leq T e^{C (1 + \delta^4 + \frac{1}{\epsilon^4}) T} \bigg[ \Bigl(C_3(1+\delta^2+\delta^4) 
+ \frac{C_2 + C_4}{\epsilon^4} \Bigr) \tau \notag \\
&\qquad + C_0 (1 + \delta^2 + \delta^4) h^{2 \min{(r,s-1)}} 
+ \Bigl( C_0 + \frac{C_0^2+C_0}{\epsilon^4} \Bigr) h^{2 \min{(r+1,s)}} \bigg] \notag \\
& \leq C(\delta, \epsilon, T, \|X\|, C_0, C_2, C_3,C_4) \bigl[\tau + h^{2 \min{(r,s-1)}}\bigr]. 
\notag
\end{align}
\end{theorem}

\begin{proof}
We write $e^n=u(t_n)-u_h^n = \eta^n + \xi^n$ where
\begin{align*}
\eta^n := u(t_n) - P_h u(t_n) \quad \text{and} \quad \xi^n := P_h u(t_n) - u_h^n,  \quad n = 0,1,2,...,N.
\end{align*}
It follows from \eqref{var_solu} that for all $t_n$ ($n \geq 0$) there holds $\P$-almost surely
\begin{align}
\label{derrest:1}
&\bigl(u(t_{n+1}), v_h) - (u(t_n), v_h\bigr)+\int_{t_n}^{t_{n+1}} \bigl(\nabla u(s), \nabla v_h\bigr) \, ds \\
&\qquad +  \frac{1}{\epsilon^2} \int_{t_n}^{t_{n+1}} \bigl(f(u(s)), v_h\bigr) \, ds 
+ \frac{\delta^2}{2} \int_{t_n}^{t_{n+1}} \bigl(B \nabla u(s), \nabla v_h\bigr) \, ds \notag \\
&\qquad + \frac{\delta^2}{2} \int_{t_n}^{t_{n+1}} \bigl((\mbox{div} B - b) 
\cdot \nabla u(s), v_h\bigr) \, ds \notag \\
&= \delta \int_{t_n}^{t_{n+1}} \bigl(\nabla u(s) \cdot X, v_h\bigr) \, dW(s) 
\qquad\qquad  \forall \, v_h \in V_h. \notag
\end{align}
Subtracting \eqref{dfem} from \eqref{derrest:1} and using the decomposition of $e^{n+1}$, 
we obtain the following error equation:
\begin{align} \label{derrest:2}
&\bigl(\xi^{n+1} - \xi^n, v_h\bigl) = -\bigl(\eta^{n+1} - \eta^n, v_h\bigr) 
- \int_{t_n}^{t_{n+1}} \bigl(\nabla u(s) - \nabla u_h^{n+1}, \nabla v_h\bigr) \, ds \\
&\qquad -\frac{1}{\epsilon^2} \int_{t_n}^{t_{n+1}} \bigl(f(u(s))-f^{n+1}, v_h\bigr) \, ds \notag \\
   &\qquad - \frac{\delta^2}{2} \int_{t_n}^{t_{n+1}} \bigl(B (\nabla u(s) - \nabla u_h^{n+1}), 
\nabla v_h\bigr) \, ds \notag \\
   &\qquad - \frac{\delta^2}{2} \int_{t_n}^{t_{n+1}} \bigl( (\mbox{div} B - b) 
\cdot (\nabla u(s) - \nabla u_h^{n}), v_h \bigr) \, ds \notag \\
   &\qquad + \delta  \int_{t_n}^{t_{n+1}} \bigl((\nabla u(s) - \nabla u_h^n) \cdot X, v_h\bigr) \, dW(s) \qquad \qquad \forall \, v_h \in V^r_h. \notag
\end{align}
Setting $v_h = \xi^{n+1} (\omega)$ in \eqref{derrest:2}, we get $\P$-almost surely
\begin{align} \label{derrest:3}
(\xi^{n+1} - \xi^n, \xi^{n+1}) &= -(\eta^{n+1} - \eta^n, \xi^{n+1}) 
- \int_{t_n}^{t_{n+1}} \bigl(\nabla u(s) - \nabla u_h^{n+1}, \nabla \xi^{n+1} \bigr) \, ds \\
&\qquad - \frac{1}{\epsilon^2} \int_{t_n}^{t_{n+1}} \bigl(f(u(s)) - f^{n+1}, \xi^{n+1} \bigr) \, ds \notag \\
   &\qquad - \frac{\delta^2}{2} \int_{t_n}^{t_{n+1}} \bigl(B (\nabla u(s) - \nabla u_h^{n+1}), 
\nabla \xi^{n+1}\bigr) \, ds \notag \\
   &\qquad - \frac{\delta^2}{2} \int_{t_n}^{t_{n+1}} \bigl( (\mbox{div} B - b) 
\cdot (\nabla u(s) - \nabla u_h^{n}), \xi^{n+1} \bigr) \, ds \notag \\
   &\qquad + \delta  \int_{t_n}^{t_{n+1}} \bigl((\nabla u(s) - \nabla u_h^n) \cdot X, \xi^{n+1} \bigr) \, dW(s), \notag \\
   & := T_1 + T_2 + T_3 + T_4 + T_5 + T_6. \notag 
\end{align}

By the elementary identity 
$(a-b)a = \frac{1}{2} (a^2-b^2) + \frac{1}{2} (a-b)^2$, we get 
\begin{align} \label{derrest:4}
 \E \left[(\xi^{n+1} - \xi^n, \xi^{n+1}) \right] 
&= \frac{1}{2} \E \left[ \|\xi^{n+1}\|_{L^2(\D)}^2 - \|\xi^{n}\|_{L^2(\D)}^2 \right] \\
&\quad +  \frac{1}{2} \E \left[ \| \xi^{n+1} - \xi^n \|^2_{L^2(\D)} \right]. \notag
\end{align}

Next, we bound the right-hand side of \eqref{derrest:3}. First, since $P_h$ is the 
$L^2$-projection operator, we have $\E \left[ T_1 \right] = 0$.
For the second term on the right-hand side of \eqref{derrest:3}, we have by Lemma~\ref{lem:e3} 
and Young's inequality that
\begin{align} \label{derrest:6}
\E \left[ T_2 \right] &= - \E \left[ \int_{t_n}^{t_{n+1}} (\nabla u(s) - \nabla u(t_{n+1}), \nabla \xi^{n+1}) \, ds \right] \\
&\quad 
- \E\left[\int_{t_n}^{t_{n+1}} (\nabla \eta^{n+1} + \nabla \xi^{n+1}, \nabla \xi^{n+1}) \, ds \right] \notag \\
&\leq C \E \left[ \int_{t_n}^{t_{n+1}} \| \nabla u(s) - \nabla u(t_{n+1}) \|_{L^2(\D)}^2 \, ds \right] 
+ \frac{3}{16} \E \left[ \| \nabla \xi^{n+1} \|^2_{L^2(\D)} \right] \tau \notag \\
&\quad + \E \left[ \| \nabla \eta^{n+1} \|^2_{L^2(\D)} \right] \tau 
+ \frac{1}{4} \E \left[ \| \nabla \xi^{n+1} \|^2_{L^2(\D)} \right] \tau \notag \\
&\quad - \E \left[ \| \nabla \xi^{n+1} \|^2_{L^2(\D)} \right] \tau \notag \\
&\leq C_3 \tau^2 + \E \left[ \| \nabla \eta^{n+1} \|^2_{L^2(\D)} \right] \tau 
- \frac{9}{16} \E \left[ \| \nabla \xi^{n+1} \|^2_{L^2(\D)} \right] \tau. \notag
\end{align}

In order to estimate the third term on the right-hand side of \eqref{derrest:3}, we write 
\begin{align}
\label{derrest:7}
-\bigl( f(u(s)) - f^{n+1}, \xi^{n+1} \bigr) &= -\bigl( f(u(s)) - f(u(t_{n+1})), \xi^{n+1} \bigr) \\
   &\quad - \bigl( f(u(t_{n+1}) - f(P_h u(t_{n+1})), \xi^{n+1} \bigr) \notag \\
   &\quad -  \bigl( f(P_h u(t_{n+1})) - f^{n+1}, \xi^{n+1} \bigr). \notag
\end{align}
By Lemma~\ref{lem:e4}, we obtain
\begin{align}
\label{derrest:8}
&- \E \left[ \bigl( f(u(s)) - f(u(t_{n+1})), \xi^{n+1} \bigr) \right] \\
&\qquad \leq \frac{1}{2C_{\dagger}} \E \left[ \| f(u(s)) - f(u(t_{n+1})) \|^2_{L^2(\D)} \right]  
+ \frac{C_{\dagger}}{2} \E \left[ \|\xi^{n+1}\|_{L^2(\D)}^2 \right] \notag \\
&\qquad \leq \frac{C_4}{2C_{\dagger}} \tau 
+ \frac{C_{\dagger}}{2} \E \left[ \|\xi^{n+1}\|_{L^2(\D)}^2 \right]. \notag
\end{align}

Next, by \eqref{A:Reg} and \eqref{Ph1}--\eqref{Ph2}, we have
\begin{align}
\label{derrest:9}
- \E &\left[ \bigl( f(u(t_{n+1}) - f(P_h u(t_{n+1})), \xi^{n+1} \bigr) \right] \\
&= - \E \left[ \bigl( \eta^{n+1} (u(t_{n+1})^2+u(t_{n+1})P_h u(t_{n+1})+P_h u(t_{n+1})^2-1),\xi^{n+1} \bigr) \right] \notag\\
& \leq \frac{1}{4 C_{\dagger}} \E \Bigl[ \|u(t_{n+1})^2+u(t_{n+1})P_h u(t_{n+1})+P_h u(t_{n+1})^2-1\|_{L^{\infty}(\D)}^2 \notag \\
&\quad \times \|\eta^{n+1}\|_{L^2(\D)}^2 \Bigr] + C_{\dagger} \E \left[ \|\xi^{n+1}\|_{L^2(\D)}^2 \right] \notag\\
& \leq \frac{1}{4 C_{\dagger}} \left( \E \left[ \|u(t_{n+1})^2+u(t_{n+1})P_h u(t_{n+1})
+P_h u(t_{n+1})^2-1\|_{L^{\infty}(\D)}^{3} \right] \right)^{\frac23} \notag \\
& \quad \times \Bigl( \E \left[\|\eta^{n+1}\|_{L^2(\D)}^{6} \right] \Bigr)^{\frac13} 
+ C_{\dagger} \E \left[ \|\xi^{n+1}\|_{L^2(\D)}^2 \right] \notag \\
& \leq \frac{C}{4 C_{\dagger}} \Bigl( \E \left[ \|P_h u(t_{n+1})\|_{L^\infty(\D)}^{6}
+\|u(t_{n+1})\|_{L^\infty(\D)}^{6}+|D|^{3} \right] \Bigr)^{\frac{2}{3}}  \notag \\
&\quad \times \Bigl( \E \left[\|\eta^{n+1}\|_{L^2(\D)}^{6} \right] \Bigr)^{\frac13} 
+ C_{\dagger} \E \left[ \|\xi^{n+1}\|_{L^2(\D)}^2 \right] \notag\\
&\leq \frac{C_0}{4 C_{\dagger}} \left( \E \left[\|\eta^{n+1}\|_{L^2(\D)}^{6} \right] \right)^{\frac13} 
+ C_{\dagger}  \E \left[ \|\xi^{n+1}\|_{L^2(\D)}^2 \right].\notag
\end{align}

When $f^{n+1} = (u_h^{n+1})^3 - u_h^{n+1}$, the last term on the right-hand 
side of \eqref{derrest:7} can be bounded by the monotonicity property 
\begin{align}
\label{derrest:9_1}
- \E \left[ \bigl( f(P_h u(t_{n+1})) - f^{n+1}, \xi^{n+1} \bigr) \right] \leq \E \left[ \| \xi^{n+1} \|^2_{L^2(\D)} \right].
\end{align}

When $f^{n+1} = (u_h^{n+1})^3 - u_h^{n}$, the last term on the right-hand side of 
\eqref{derrest:7} results in an extra term $\left( u_h^{n+1} - u_h^n, \xi^{n+1} \right)$. 
It can be controlled by using Lemma~\ref{lem:e2} as follows
\begin{align}
\label{derrest:10}
\E &\left[ \bigl( u_h^{n+1} - u_h^n, \xi^{n+1} \bigr) \right] 
= - \E \left[ \bigl( (\eta^{n+1} - \eta^n) + (\xi^{n+1} - \xi^n), \xi^{n+1} \bigr) \right] \\
&\qquad + \E \left[ \bigl( u(t_{n+1}) - u(t_n), \xi^{n+1} \bigr) \right] \notag \\
&\quad \leq \frac{1}{2 C_{\dagger}} \E \left[ \|\eta^{n+1} - \eta^n\|^2_{L^2(\D)} \right] 
+ C_{\dagger} \E \left[ \|\xi^{n+1}\|^2_{L^2(\D)} \right] + \frac{C_2}{C_{\dagger}} \tau \notag \\
&\qquad - \E \left[ \|\xi^{n+1}\|^2_{L^2(\D)} \right] + \frac{1}{4 C_{\dagger}} 
\E \left[ \|\xi^{n}\|^2_{L^2(\D)} \right] + C_{\dagger} \E \left[ \|\xi^{n+1}\|^2_{L^2(\D)} \right]. \notag
\end{align}
Combining \eqref{derrest:7}--\eqref{derrest:10}, we choose $C_{\dagger} = 1$ 
when $f^{n+1} = (u_h^{n+1})^3 - u_h^{n+1}$ to obtain
\begin{align} \label{derrest:11}
\E \left[ T_3 \right] \leq \frac{C_4}{\epsilon^2} \tau^{2} 
+ \frac{1}{\epsilon^2} \E \left[ \|\xi^{n+1}\|_{L^2(\D)}^2 \right] \tau 
+ \frac{C_0}{\epsilon^2} \left( \E \left[\|\eta^{n+1}\|_{L^2(\D)}^{4} \right] \right)^{\frac12} \tau,
\end{align}
and choose $C_{\dagger} = \epsilon^2$ when $f^{n+1} = (u_h^{n+1})^3 - u_h^{n}$ to obtain
\begin{align} \label{derrest:12}
\E \left[ T_3 \right] &\leq \frac{C_2 + C_4}{\epsilon^4} \tau^2 
+ \E \left[ \|\xi^{n+1}\|_{L^2(\D)}^2 \right] \tau 
+ \frac{1}{\epsilon^4} \E \left[ \|\xi^{n}\|_{L^2(\D)}^2 \right] \tau \\
&\quad 
+ \frac{C_0}{\epsilon^4} \left( \E \left[\|\eta^{n+1}\|_{L^2(\D)}^{4} \right] \right)^{\frac12} \tau 
+ \frac{1}{\epsilon^4} \E \left[ \|\eta^{n+1}\|_{L^2(\D)}^2 + \|\eta^{n}\|_{L^2(\D)}^2 \right] \tau. \notag
\end{align}

Similar to the estimate of $T_2$, the fourth and fifth terms on the right-hand side of 
\eqref{derrest:3} can also be bounded by Lemma~\ref{lem:e3} and Young's inequality as follows
\begin{align}
\label{derrest:13}
\E \left[ T_4 \right] &= - \frac{\delta^2}{2} \E \left[ \int_{t_n}^{t_{n+1}} 
\bigl( (\nabla u(s) - \nabla u(t_{n+1}) \cdot X, \nabla \xi^{n+1} \cdot X \bigr) \, ds \right] \\
   &\quad - \frac{\delta^2}{2} \E \left[ \int_{t_n}^{t_{n+1}} \bigl( (\nabla \eta^{n+1} 
+ \nabla \xi^{n+1}) \cdot X, \nabla \xi^{n+1} \cdot X \bigr) \, ds \right] \notag \\
   &\leq \delta^4 \|X\|^4_{C(\bar{\D})} C_3 \tau^2 + \frac{2}{16} 
\E \left[ \| \nabla \xi^{n+1} \|^2_{L^2(\D)} \right] \tau \notag \\
   &\quad + \delta^4 \|X\|^4_{C(\bar{\D})}\E \left[ \| \nabla \eta^{n+1} \|^2_{L^2(\D)} \right] \tau 
- \frac{\delta^2}{2} \E \left[ \| \nabla \xi^{n+1} \cdot X \|^2_{L^2(\D)} \right] \tau, \notag
\end{align}
and
\begin{align}
\label{derrest:14}
\E \left[ T_5 \right] &= - \frac{\delta^2}{2} \E \left[ \int_{t_n}^{t_{n+1}} 
\Bigl( (\mbox{div} B - b) \cdot (\nabla u(s) - \nabla u(t_{n})), \xi^{n+1} \Bigr) \, ds \right] \\
   &\quad - \frac{\delta^2}{2} \E \left[ \int_{t_n}^{t_{n+1}} 
\Bigl( (\mbox{div} B - b) \cdot (\nabla \eta^{n} + \nabla \xi^{n}), \xi^{n+1} \Bigr) \, ds \right] \notag \\
   &\leq \|X\|^4_{C^1(\bar{\D})} C_3 \tau^2 + (1 + \|X\|^4_{C^1(\bar{\D})}) \delta^4 
\E \left[ \|\xi^{n+1}\|_{L^2(\D)}^2 \right] \tau \notag \\
   &\quad + \|X\|^4_{C^1(\bar{\D})} \E \left[ \| \nabla \eta^{n} \|^2_{L^2(\D)} \right] \tau 
+ \frac{1}{16}  \E \left[ \| \nabla \xi^{n} \|^2_{L^2(\D)} \right] \tau. \notag
\end{align}

By the martingale property, It\^{o}'s isometry and Lemma~\ref{lem:e3}, we have
\begin{align}
\label{derrest:15}
&\E [T_6] \leq \frac{1}{2} \E \left[ \|\xi^{n+1} - \xi^n\|^2_{L^2(\D)} \right] \\
&\quad + \frac{\delta^2}{2} \E \left[ \int_{t_n}^{t_{n+1}} 
\| (\nabla u(s) - \nabla u_h^n) \cdot X \|^2_{L^2(\D)} \, ds \right] \notag \\
   &\leq \frac{1}{2} \E \left[ \|\xi^{n+1} - \xi^n\|^2_{L^2(\D)} \right] \notag \\
   &\quad + \frac{\delta^2}{2} (1 + C') \E \left[ \int_{t_n}^{t_{n+1}} 
\| (\nabla u(s) - \nabla u(t_n)) \cdot X \|^2_{L^2(\D)} \, ds \right] \notag \\
   &\quad + \frac{\delta^2}{2} \E \left[ \|(\nabla \eta^n + \nabla \xi^n) \cdot X\|^2_{L^2(\D)} \right] \tau 
+ \frac{\delta^2}{2 C'} \E \left[ \|(\nabla \eta^n + \nabla \xi^n) \cdot X\|^2_{L^2(\D)} \right] \tau \notag \\
   &\leq \frac{1}{2} \E \left[ \|\xi^{n+1} - \xi^n\|^2_{L^2(\D)} \right] 
+ \frac{\delta^2}{2} (1 + C') \|X\|^2_{C(\bar{\D})} C_3 \tau^2 \notag \\
   &\quad + \delta^2 \Bigl( \frac{1+C''}{2} + \frac{1}{C'} \Bigr) 
\|X\|^2_{C(\bar{\D})} \E \left[ \| \nabla \eta^n \|^2_{L^2(\D)} \right] \tau \notag \\
   &\quad + \frac{\delta^2}{2} \E \left[ \| \nabla \xi^{n} \cdot X \|^2_{L^2(\D)} \right] \tau 
+ \delta^2 \left( \frac{1}{2C''} + \frac{1}{C'} \right) \|X\|^2_{C(\bar{\D})} 
\E \left[ \| \nabla \xi^n \|^2_{L^2(\D)} \right] \tau. \notag
\end{align}
Now taking $C' = 16 \delta^2 \|X\|^2_{C(\bar{\D})}$ and $C'' = 8 \delta^2 \|X\|^2_{C(\bar{\D})}$ 
in \eqref{derrest:15}, we obtain an estimate for the last term on the right-hand side of \eqref{derrest:3}:
\begin{align} \label{derrest:16}
\E [T_6] &\leq \frac{1}{2} \E \left[ \|\xi^{n+1} - \xi^n\|^2_{L^2(\D)} \right] 
+ \frac{\delta^2}{2} (1 + 16 \delta^2 \|X\|^2_{C(\bar{\D})}) \|X\|^2_{C(\bar{\D})} C_3 \tau^2 \\
   &\quad + \delta^2 \Bigl( \frac{1+8 \delta^2 \|X\|^2_{C(\bar{\D})}}{2} 
+ \frac{1}{16 \delta^2 \|X\|^2_{C(\bar{\D})}} \Bigr) \|X\|^2_{C(\bar{\D})} 
\E \left[ \| \nabla \eta^n \|^2_{L^2(\D)} \right] \tau \notag \\
   &\quad + \frac{\delta^2}{2} \E \left[ \| \nabla \xi^{n} \cdot X \|^2_{L^2(\D)} \right] \tau 
+ \frac{2}{16} \E \left[ \| \nabla \xi^n \|^2_{L^2(\D)} \right] \tau. \notag
\end{align}

Taking expectation on \eqref{derrest:3} and combining estimates \eqref{derrest:4}--\eqref{derrest:6}, 
\eqref{derrest:11}--\eqref{derrest:14} and \eqref{derrest:16}, summing over 
$n = 0, 1, 2, ..., l-1$ with $1 \leq l \leq N$, and using the properties of the $L^2$-projection 
and the regularity assumption \eqref{A:Reg}, we obtain
\begin{align} \label{derrest:17}
&\left[ \frac{1}{2} - \left( \frac{1}{\epsilon^2} + \delta^4 \right) \tau \right] 
\E \left[ \| \xi^l \|^2_{L^2(\D)} \right] 
+ \frac{1}{4} \E \left[ \tau \sum_{n = 1}^l \| \nabla \xi^n \|^2_{L^2(\D)} \right] \\
   &\quad \leq \frac{1}{2} \E \left[ \| \xi^0 \|^2_{L^2(\D)} \right] 
+ \frac{\delta}{2} \tau \E \left[ \| \nabla \xi^0 \cdot X \|^2_{L^2(\D)} \right] 
+ \frac{3}{16} \tau \E \left[ \| \nabla \xi^0 \|^2_{L^2(\D)} \right] \notag \\
   &\qquad + \left( C_3(1+\delta^2+\delta^4) + \frac{C_4}{\epsilon^2} \right) T \tau \notag \\
   &\qquad + C_0 (1+\delta^2+\delta^4) T h^{2 \min{(r,s-1)}} 
+ \frac{C_0^2}{\epsilon^2} T h^{2 \min{(r+1,s)}} \notag \\
   &\qquad + \left( \frac{1}{\epsilon^2} + \delta^4 \right) 
\E \left[ \tau \sum_{n=0}^{l-1} \| \xi^n \|^2_{L^2(\D)} \right], \notag
\end{align}
when $f^{n+1} = (u_h^{n+1})^3 - u_h^{n+1}$, and
\begin{align} \label{derrest:18}
&\left[ \frac{1}{2} - \left( 1 + \delta^4 \right) \tau \right] \E \left[ \| \xi^l \|^2_{L^2(\D)} \right] 
+ \frac{1}{4} \E \left[ \tau \sum_{n = 1}^l \| \nabla \xi^n \|^2_{L^2(\D)} \right] \\
   &\quad \leq \frac{1}{2} \E \left[ \| \xi^0 \|^2_{L^2(\D)} \right] 
+ \frac{\delta}{2} \tau \E \left[ \| \nabla \xi^0 \cdot X \|^2_{L^2(\D)} \right] 
+ \frac{3}{16} \tau \E \left[ \| \nabla \xi^0 \|^2_{L^2(\D)} \right] \notag \\
   &\quad \quad + \left( C_3(1+\delta^2+\delta^4) + \frac{C_2 + C_4}{\epsilon^4} \right) T \tau \notag \\
   &\qquad + C_0 (1+\delta^2+\delta^4) T h^{2 \min{(r,s-1)}} 
+ \frac{C_0^2+C_0}{\epsilon^4} T h^{2 \min{(r+1,s)}} \notag \\
   &\qquad + \left( 1+ \frac{1}{\epsilon^4} + \delta^4 \right) 
\E \left[ \tau \sum_{n=0}^{l-1} \| \xi^n \|^2_{L^2(\D)} \right], \notag
\end{align}
when $f^{n+1} = (u_h^{n+1})^3 - u_h^{n}$, here we have not explicitly tracked the constant 
$\| X \|_{C^1(\bar{\D})}$ or $\| X \|_{C(\bar{\D})}$ and some other constants 
in \eqref{derrest:17} and \eqref{derrest:18}.

Finally, estimates \eqref{derrest1} and \eqref{derrest2} follow from \eqref{derrest:17}--\eqref{derrest:18}, 
the discrete Gronwall's inequality, the $L^2$-projection properties, the fact $\xi^0 = 0$ 
and the triangle inequality.
\end{proof}

\begin{remark}
Error estimates in Theorem \ref{thm:derrest} remain unchanged if we consider the modified scheme 
\eqref{dfem2}. In fact, we only need to check the fifth term on the right-hand side of \eqref{derrest:3}:
\begin{align} \label{derrest:14_2}
\E \left[ T_5 \right] &= - \frac{\delta^2}{2} \E \left[ \int_{t_n}^{t_{n+1}} 
\bigl( (\mbox{\rm div} B - b) \cdot (\nabla u(s) - \nabla u(t_{n+1})), \xi^{n+1} \bigr) \, ds \right] \\
   &\quad - \frac{\delta^2}{2} \E \left[ \int_{t_n}^{t_{n+1}} 
\bigl( (\mbox{\rm div} B - b) \cdot (\nabla \eta^{n+1} + \nabla \xi^{n+1}), \xi^{n+1} \bigr) \, ds \right] \notag \\
   &\leq \|X\|^4_{C^1(\bar{\D})} C_3 \tau^2 + (1 + \|X\|^4_{C^1(\bar{\D})}) \delta^4 
\E \left[ \|\xi^{n+1}\|_{L^2(\D)}^2 \right] \tau \notag \\
   &\quad + \|X\|^4_{C^1(\bar{\D})} \E \left[ \| \nabla \eta^{n+1} \|^2_{L^2(\D)} \right] \tau 
+ \frac{1}{16} \E \left[ \| \nabla \xi^{n+1} \|^2_{L^2(\D)} \right] \tau. \notag
\end{align}
Hence the third term $\frac{3}{16} \tau \E \left[ \| \nabla \xi^0 \|^2_{L^2(\D)} \right]$ on 
the right-hand side of \eqref{derrest:17} or \eqref{derrest:18} is replaced by 
$\frac{1}{8} \tau \E \left[ \| \nabla \xi^0 \|^2_{L^2(\D)} \right]$.
\end{remark}

\begin{remark}
Error estimates \eqref{derrest1} and \eqref{derrest2} are optimal with respect to $\tau$, and 
optimal in the $H^1$-norm but suboptimal in the $L^2$-norm with respect to $h$. From the proof 
of Theorem \ref{thm:derrest}, the suboptimal estimate in the $L^2$-norm is caused by gradient-type 
noises, i.e., the existence of $T_4$, $T_5$ and $T_6$ on the right-hand side of \eqref{derrest:3}. 
The proof in Theorem \ref{thm:derrest} relies on the $p$-th moment estimate \eqref{A:Reg} for the 
strong solution. Convergence rates are expected to be lower under weaker regularity assumptions. 
We also note that the error bounds depend exponentially on $1/\epsilon$, which seems to be pessimistic. 
However, this is the case even in the deterministic case (i.e., $\delta=0$) unless the standard error
estimate technique is replaced by a much involved nonstandard technique as used in \cite{FP2003}. 
We intend to address this issue in a future work. 
\end{remark}

\section{Numerical experiments}\label{sec:numer}
In this section we present a couple two-dimensional numerical experiments to gauge the performance of the 
proposed fully discrete finite element methods with $r=1$, i.e., $V_h$ is the linear finite element space.
We also numerically examine the influence of noises on the dynamics of the numerical interfaces.

We consider the SPDE \eqref{sac} on the square domain $\D = [-0.5,0.5]^2$ with two different initial 
conditions, and in both tests we take $X(x) = \varphi(x)[x_1+x_2,\ x_1-x_2]^{T}$, where
\begin{align*}
\varphi(x) := \begin{cases}
e^{-\frac{0.001}{0.09^2-|x|^2}}, & \text{if}\ 
|x|<0.3,\\
0, & \text{if}\ 
|x|\geq0.3.\\
\end{cases}
\end{align*} 
For both tests, we take the Brownian motion step to be $1\times10^{-4}$ and compute $M=500$ Monte Carlo 
realizations. In this section, we use a new notion $u^{\delta,\epsilon,h,\tau}(\omega)$ to denote the numerical 
solution, where $\delta > 0$ controls the noise intensity, $\epsilon > 0$ is the diffuse interface width, 
$h$ and $\tau$ are the mesh sizes, and $\omega$ is a discrete sample.

Next, we describe the algorithm which we use to solve the discrete problem \eqref{dfem}. 
Let $N_h = \text{dim} V_h$ and $\{ \psi_i \}_{i=1}^{N_h}$ be the nodal basis of $V_h$. 
Denote by $\bf{u^{n+1}}$ the coefficient vector corresponding to the discrete solution 
$u^{n+1}_h = \sum_{i = 0}^{N_h} u^{n+1}_i \psi_i$ at time $t_{n+1} = (n+1) \tau$. Suppose 
$f^{n+1} = (u^{n+1}_h)^3 - u^{n}_h$ (the other case is similar), \eqref{dfem} is then equivalent to
\begin{align}\label{eq4.1}
&\left[ \bf{M} + \tau \left( \bf{A} + \frac{\delta^2}{2} \bf{A_X} \right) \right] \bf{u^{n+1}} 
+ \frac{\tau}{\epsilon^2} \bf{N} (\bf{u^{n+1}}) \\
&\qquad = \left( 1 + \frac{\tau}{\epsilon^2} \right) \bf{M} \bf{u^n} 
- \frac{\tau \delta^2}{2} \bf{C_1} \bf{u^n} + \delta \Delta \bf{W_{n+1}} \bf{C_2} \bf{u^n}, \notag
\end{align}
where $\bf{M}, \bf{A}$ are the standard mass and stiffness matrices, respectively, $\bf{A_X}$ is the 
weighted stiffness matrix whose $(i,j)$ component is $(\nabla \psi_j \cdot X, \nabla \psi_i \cdot X)$, 
$\bf{N} (\bf{u^{n+1}})$ is the nonlinear term, $(\bf{C_1})_{ij} = ((\mbox{div} B - b) \cdot \nabla \psi_j, 
\psi_i)$, $(\bf{C_2})_{ij} = (\nabla \psi_j, \psi_i)$ and $\bf{W}$ is the discrete Brownian motion 
with increments $\Delta \bf{W_{n+1}} = W_{n+1} - W_{n}$. Since we want to generate as many samples 
as possible in order to recover quantities of statistical interests, it will be expensive to use 
Newton's method to solve \eqref{eq4.1}. Instead, we employ a cheaper fixed point iteration, although 
it may converge slower than Newton's method for a particular realization. However, if it is utilized well, 
more efficient overall solution algorithm can be designed when combined with the direct linear solver. 

Notice that the first coefficient matrix on the left-hand side of \eqref{eq4.1} is independent of 
$n$ and $\omega$, compute its Cholesky factorization
\[
\bf{M} + \tau \left( \bf{A} + \frac{\delta^2}{2} \bf{A_X} \right) = \bf{L} \bf{L'}.
\]
$\bf{L}$ will be stored and re-use for every Monte Carlo realization and at every time step. 
Only backward and forward substitutions are needed to execute the fixed point iteration,
hence, resulting in considerable amount of computational saving. From this point of view, the 
algorithm based on the fixed point iteration outperforms that based on Newton's method, especially 
for large sample size $M$. Indeed, we observed in our experiment that the fixed point iteration 
is much more efficient.

\vspace{10pt}
\noindent $\bf{Test\, 1.}$ 
The initial condition $u_0$ is chosen as
\begin{equation*}
 u_0(x) = \tanh\Bigl( \frac{d(x)}{\sqrt{2}\eps} \Bigr),
 \end{equation*}
where $\tanh (x) = (e^x - e^{-x}) / (e^x + e^{-x})$ and $d(x)$ represents the signed distance 
between the point $x$ and the ellipse
 \begin{equation*}
 \frac{x_1^2}{0.04}+\frac{x_2^2}{0.01}=1.
 \end{equation*}

In Table~\ref{tab_1}, we plot the expected values of the $L^{\infty}([0,T],||\cdot||_{L_2(\D)})$-norm 
of the errors and rates of convergence of the time discretization for varying $\tau$ with the fixed 
parameters $\delta=1$ and $\epsilon = 0.1$. The numerical results confirm the theoretical result 
of Theorem~\ref{thm:derrest}.

\begin{table}[htbp]
\begin{center}
\begin{tabular}{|c|c|c|}
\hline
        & Expected values of error      & Order of convergence \\ \hline
$\tau$=0.008 & 0.09895 & \\ \hline
$\tau$=0.004 & 0.06557 & 0.5937\\ \hline
$\tau$=0.002 & 0.04472 & 0.5521\\ \hline
$\tau$=0.001 & 0.03136 & 0.5120\\ \hline
\end{tabular}
\caption{Computed time discretization errors and convergence rates}
\label{tab_1}
\end{center}
\end{table}

In Figure~\ref{fig_t1_1}--\ref{fig_t1_3}, we display a few snapshots of the zero-level set of 
the expected value of the numerical solution 
\[
\bar{u}^{\delta,\epsilon,h,\tau} = \frac{1}{M} \sum_{i=1}^{M} u^{\delta,\epsilon,h,\tau} (\omega_i)
\]
at several time points with $\epsilon = 0.01$, and three different noise intensity parameters 
$\delta = 0.1, 1, 10$. We observe that the shape of the zero-level set of the expected value of the 
numerical solution undergoes more changes as $\delta$ increases.

\begin{figure}[thbp]
\centering
\includegraphics[height=1.8in,width=2.4in]{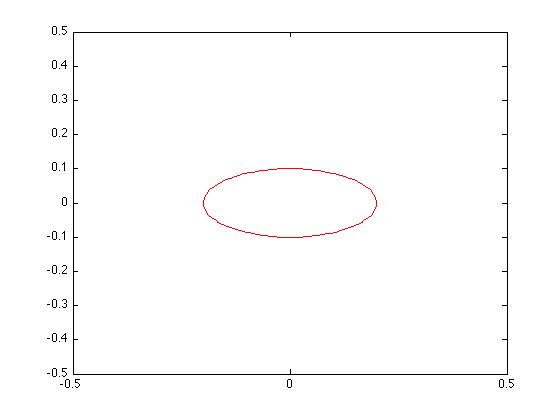}
\includegraphics[height=1.8in,width=2.4in]{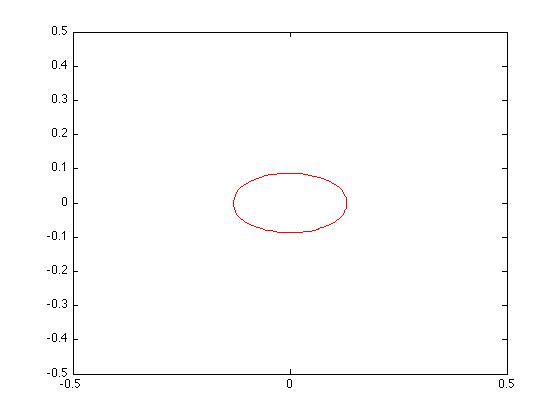}

\includegraphics[height=1.8in,width=2.4in]{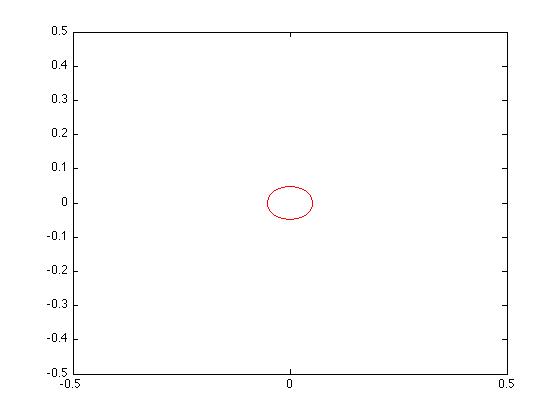}
\includegraphics[height=1.8in,width=2.4in]{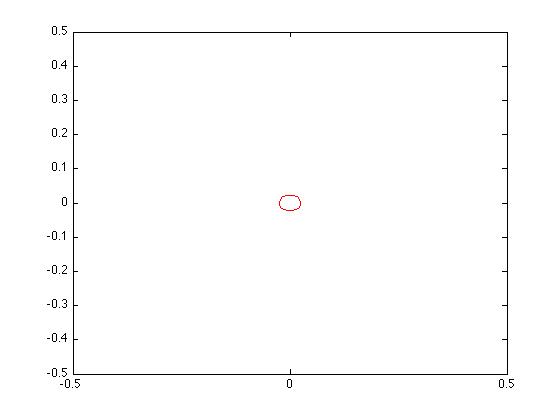}
\caption{Snapshots of the zero-level set of $\bar{u}^{\delta,\epsilon,h,\tau} $ at time 
$t=0, 0.020, 0.040, 0.043$ with $\delta=0.1$ and $\epsilon = 0.01$.}
\label{fig_t1_1}
\end{figure}

\begin{figure}[thbp]
\centering
\includegraphics[height=1.8in,width=2.4in]{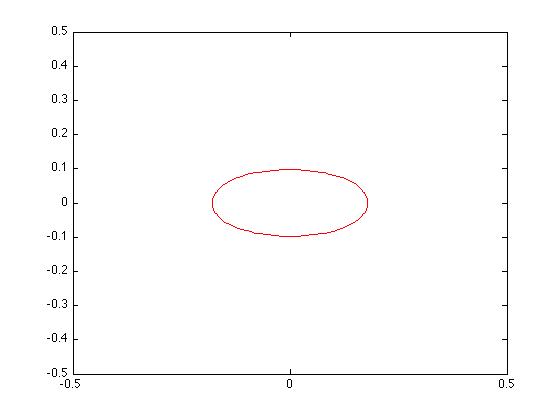}
\includegraphics[height=1.8in,width=2.4in]{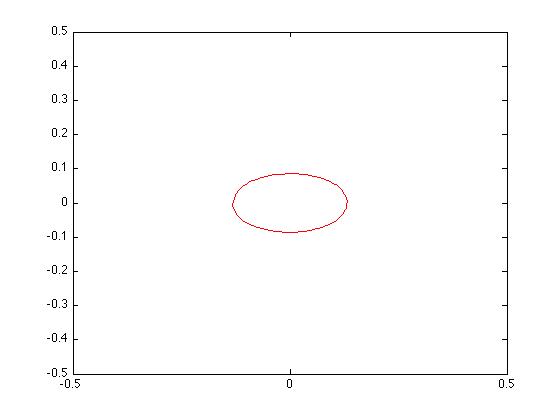}

\includegraphics[height=1.8in,width=2.4in]{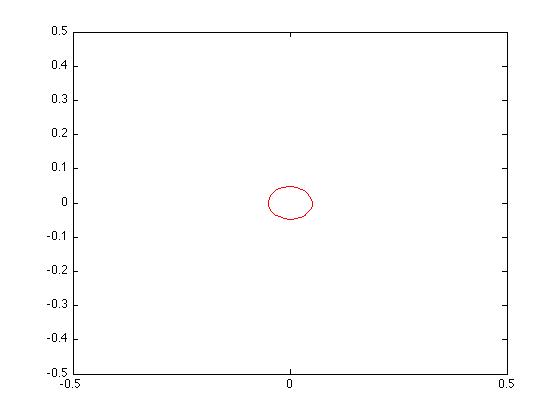}
\includegraphics[height=1.8in,width=2.4in]{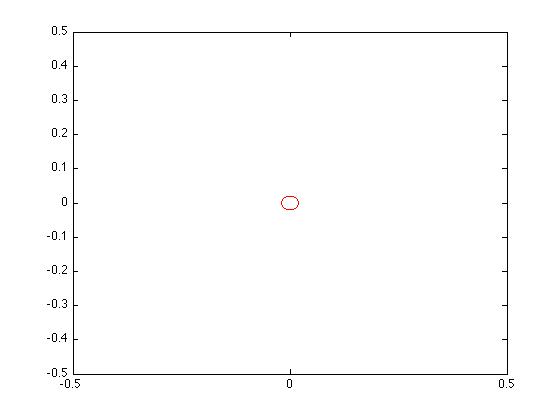}
\caption{Snapshots of the zero-level set of $\bar{u}^{\delta,\epsilon,h,\tau} $ at time 
$t=0.005, 0.020, 0.040, 0.043$ with $\delta=1$ and $\epsilon = 0.01$.}
\label{fig_t1_2}
\end{figure}

\begin{figure}[thbp]
\centering
\includegraphics[height=1.8in,width=2.4in]{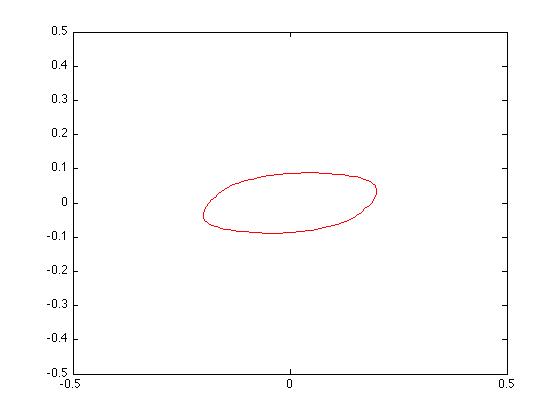}
\includegraphics[height=1.8in,width=2.4in]{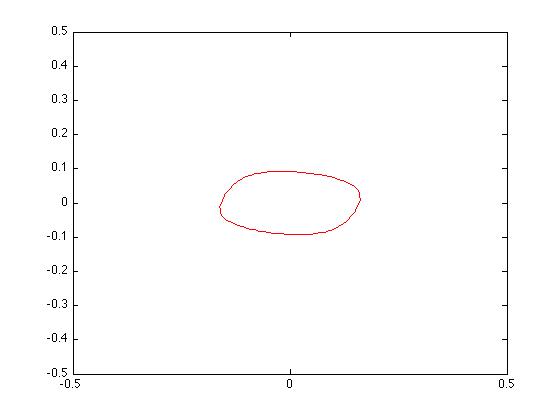}

\includegraphics[height=1.8in,width=2.4in]{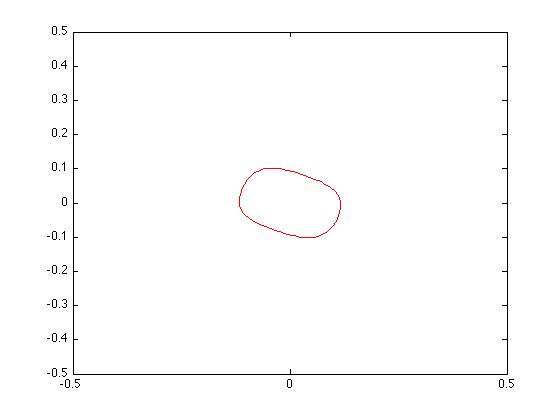}
\includegraphics[height=1.8in,width=2.4in]{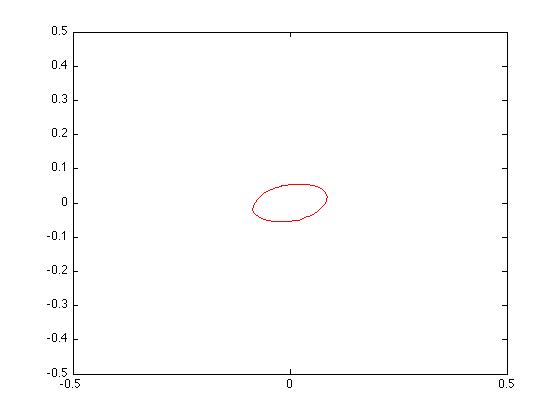}

\includegraphics[height=1.8in,width=2.4in]{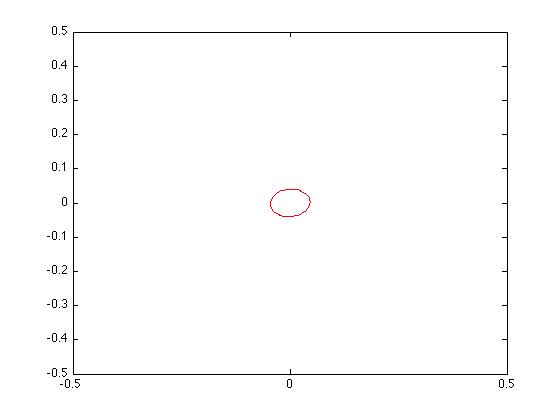}
\includegraphics[height=1.8in,width=2.4in]{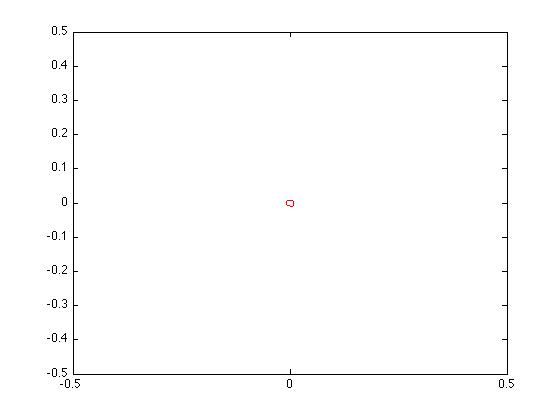}
\caption{Snapshots of the zero-level set of $\bar{u}^{\delta,\epsilon,h,\tau}$ at time 
$t=0.0025, 0.0050, 0.0100, 0.0200, 0.0250, 0.0280$ with $\delta=10$ and $\epsilon = 0.01$.}
\label{fig_t1_3}
\end{figure}

Next, we study the effect of $\epsilon$ on the evolution of the numerical interfaces. To the end, 
we fix $\delta=0.1$. In Figure~\ref{fig_t1_4}, we depict four snapshots at four fixed time points of 
the zero-level set of the expected value of the numerical solution $\bar{u}^{\delta,\epsilon,h,\tau}$ 
at three different $\epsilon = 0.01, 0.011, 0.02$. The numerical interface clearly converges to the
stochastic mean curvature flow as $\epsilon \rightarrow 0$, and it evolves faster in time for larger $\epsilon$. 
\begin{figure}[thbp]
\centering
\includegraphics[height=1.8in,width=2.4in]{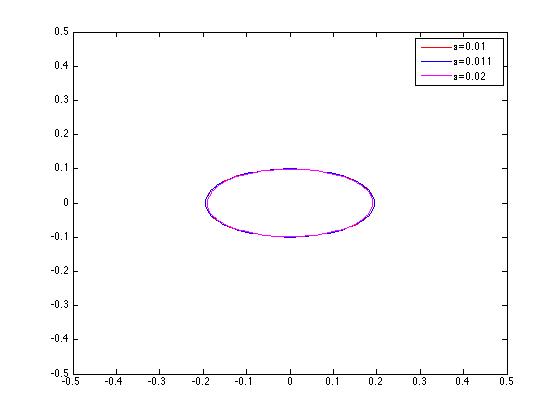}
\includegraphics[height=1.8in,width=2.4in]{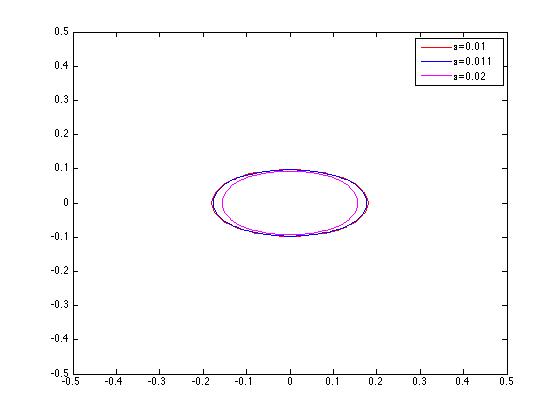}
\includegraphics[height=1.8in,width=2.4in]{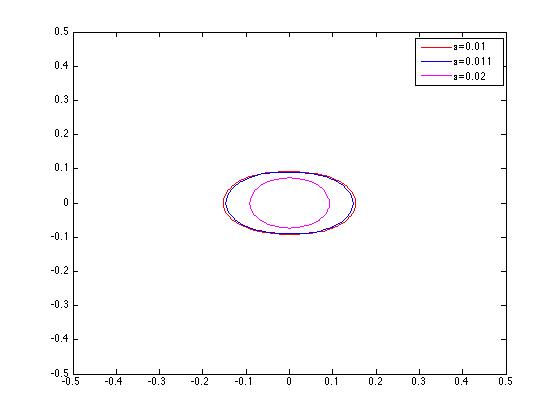}
\includegraphics[height=1.8in,width=2.4in]{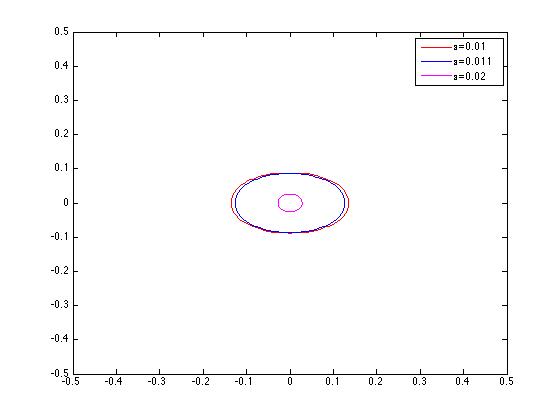}
\caption{Snapshots of the zero-level set of $\bar{u}^{\delta,\epsilon,h,\tau}$ at time 
$t=0.0010, 0.0050, 0.0125, 0.0175$ with $\delta = 0.1$ and $\epsilon = 0.01, 0.011, 0.02$.}
\label{fig_t1_4}
\end{figure}

\newpage
\noindent $\bf{Test\, 2.}$ 
First, we define
\begin{alignat*}{2}
\tanh(y)&:=\frac{e^y-e^{-y}}{e^y+e^{-y}}, &&\qquad \psi_1(x_2):= \frac{-1+\sqrt{0.8x_2+0.04}}{2},\\
\psi_2(x_2)&:= \frac{1-\sqrt{1.92x_2+0.2304}}{2}, &&\qquad \psi_3(x_2):=\frac{-1+\sqrt{-0.8x_2+0.04}}{2}, \\
\psi_4(x_2)&:= \frac{1-\sqrt{-1.92x_2+0.2304}}{2}, &&\qquad \psi_5(x_2):= -\sqrt{\frac{1-0.2451x_2^2}{0.0049}}.
\end{alignat*}
and consider the initial condition $u_0(x_1,x_2)=u_1(3x_1,3x_2)$, where $u_1(x_1,x_2)$ is defined 
by (cf. Test 2 in \cite{FL2014}):

{\small
\begin{equation*}
 u_1(x_1,x_2)=\begin{cases}
 \tanh(\frac{1}{\sqrt{2}\eps}(-\sqrt{(x_1-0.14)^2+(x_2-0.15)^2})), & \mbox{if}\ x_1> 0.14,0\leq x_2<-\frac{5}{12}(x_1-0.5),\\
 \tanh(\frac{1}{\sqrt{2}\eps}(-\sqrt{(x_1-0.14)^2+(x_2+0.15)^2})), & \mbox{if}\ x_1> 0.14,\frac{5}{12}(x_1-0.5)< x_2<0,\\
 \tanh(\frac{1}{\sqrt{2}\eps}(-\sqrt{(x_1+0.3)^2+(x_2-0.15)^2})), & \mbox{if}\ x_1<- 0.3,0\leq x_2<\frac{3}{4}(x_1+0.5),\\
 \tanh(\frac{1}{\sqrt{2}\eps}(-\sqrt{(x_1+0.3)^2+(x_2+0.15)^2})), & \mbox{if}\ x_1<- 0.3,-\frac{3}{4}(x_1+0.5)< x_2<0,\\
 \tanh(\frac{1}{\sqrt{2}\eps}(\sqrt{(x_1-0.5)^2+x_2^2}-0.39)), & \mbox{if}\ x_1>0.14,x_2\geq-\frac{5}{12}(x_1-0.5)\\
 &\ \mbox{or}\ x_2\leq\frac{5}{12}(x_1-0.5),\\
  \tanh(\frac{1}{\sqrt{2}\eps}(\sqrt{(x_1+0.5)^2+x_2^2}-0.25)), &  \mbox{if}\ x_1<-0.3,x_2\geq-\frac{3}{4}(x_1+0.5)\\
 &\ \mbox{or}\ x_2\leq-\frac{3}{4}(x_1+0.5),\\
 \tanh(\frac{1}{\sqrt{2}\eps}(|x_2|-0.15)), & \text{if}\ -0.3\leq x_1\leq0.14,\\
 &\,\, \psi_1(x_2)\leq x_1\leq\psi_2(x_2) \\
 &\,\, \mbox{and}\ \psi_3(x_2)\leq x_1\leq\psi_4(x_2),\\
 \tanh(\frac{1}{\sqrt{2}\eps}(\sqrt{(x_1-0.5)^2+x_2^2}-0.39)), & \mbox{if}\ -0.3\leq x_1\leq0.14,\,\,
  x_1\geq \psi_2(x_2)\\
 &\,\, \mbox{and}\ x_1\geq \psi_5(x_2),\\
 \tanh(\frac{1}{\sqrt{2}\eps}(\sqrt{(x_1-0.5)^2+x_2^2}-0.39)), & \mbox{if}\ -0.3\leq x_1\leq0.14,\,\,
 x_1\geq \psi_4(x_2)\\
 &\,\, \mbox{and}\ x_1\geq \psi_5(x_2),\\
 \tanh(\frac{1}{\sqrt{2}\eps}(\sqrt{(x_1+0.5)^2+x_2^2}-0.25)), &  \mbox{if}\ -0.3\leq x_1\leq0.14,\,\,
 x_1\leq \psi_1(x_2)\\
 &\,\, \mbox{and}\ x_1\leq \psi_5(x_2),\\
 \tanh(\frac{1}{\sqrt{2}\eps}(\sqrt{(x_1+0.5)^2+x_2^2}-0.25)), &  \mbox{if}\ -0.3\leq x_1\leq0.14,\,\,
 x_1\leq \psi_3(x_2)\\
 &\,\, \mbox{and}\ x_1\leq \psi_5(x_2).
 \end{cases}
 \end{equation*}
}

We note that the initial condition is not smooth due to the dumbbell shape of the zero-level set.
Nevertheless, we study the effects of $\delta$ and $\epsilon$ on the evolution of numerical interfaces. 
Figure~\ref{fig_t2_1}--\ref{fig_t2_3} display a few snapshots of the zero-level set of 
the expected value of the numerical solution $\bar{u}^{\delta,\epsilon,h,\tau}$ at several time points
with $\epsilon = 0.01$ and three different noise intensity $\delta = 0.1, 1, 10$. Similar to Test $1$, 
the zero-level set evolves faster and the shape undergoes more changes for larger $\delta$. 
Figure~\ref{fig_t2_4} plots snapshots at four fixed time point of the zero-level set 
of $\bar{u}^{\delta,\epsilon,h,\tau}$ with $\delta = 0.1$ and $\epsilon = 0.01, 0.011, 0.02$. 
Again, it suggests the convergence of the zero-level set at each time point to the stochastic 
mean curvature flow as $\epsilon \rightarrow 0$.

\begin{figure}[thbp]
\centering
\includegraphics[height=1.8in,width=2.4in]{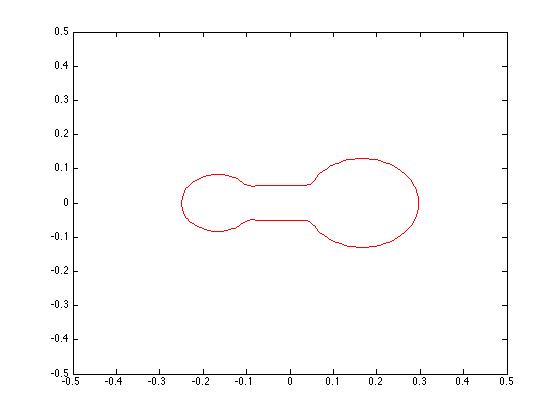}
\includegraphics[height=1.8in,width=2.4in]{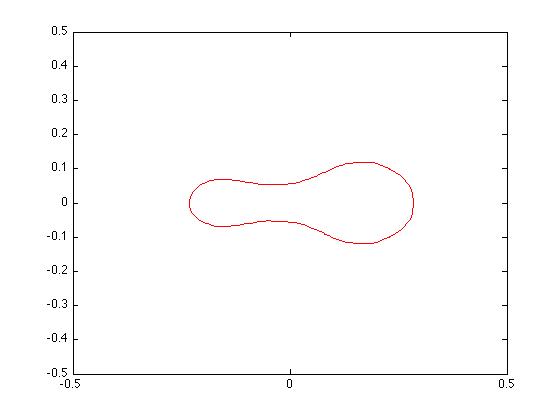}

\includegraphics[height=1.8in,width=2.4in]{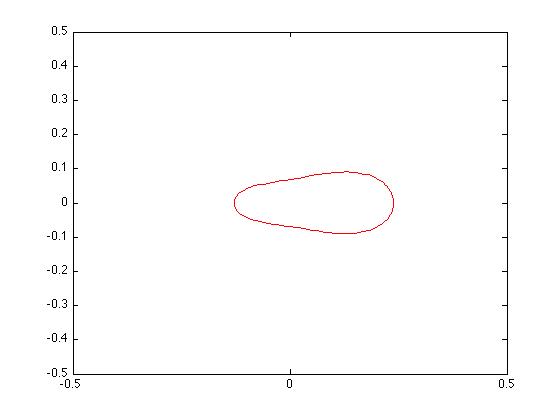}
\includegraphics[height=1.8in,width=2.4in]{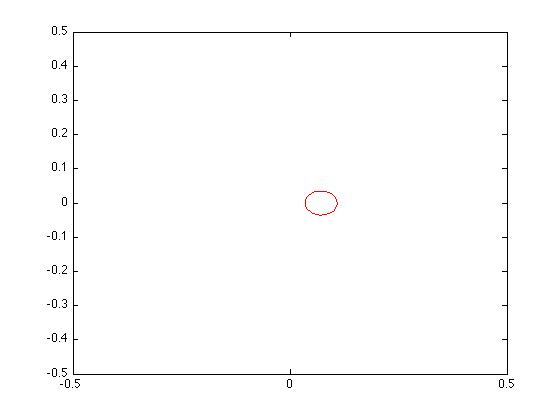}
\caption{Snapshots of the zero-level set of $\bar{u}^{\delta,\epsilon,h,\tau}$ at time 
$t=0, 0.040, 0.200, 0.456$ with $\delta=0.1$ and $\epsilon = 0.01$.}
\label{fig_t2_1}
\end{figure}

\begin{figure}[thbp]
\centering
\includegraphics[height=1.8in,width=2.4in]{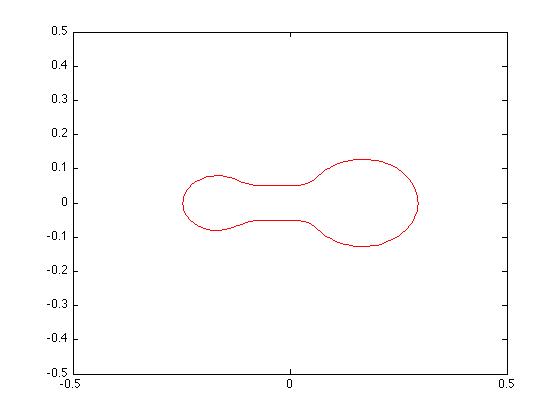}
\includegraphics[height=1.8in,width=2.4in]{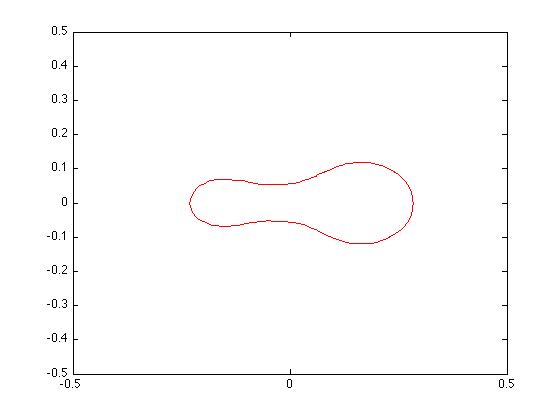}

\includegraphics[height=1.8in,width=2.4in]{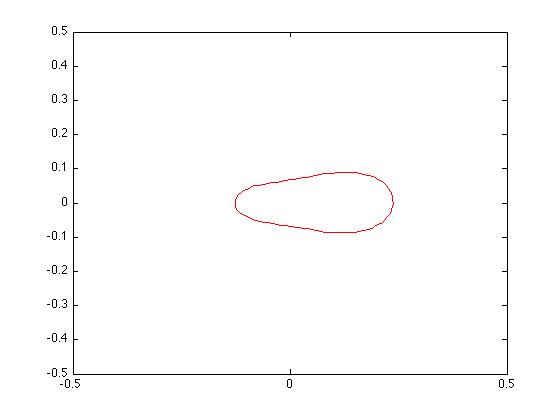}
\includegraphics[height=1.8in,width=2.4in]{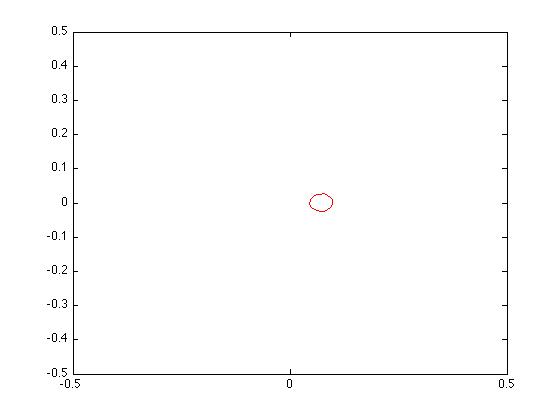}
\caption{Snapshots of the zero-level set of $\bar{u}^{\delta,\epsilon,h,\tau}$ at time 
$t=0.004, 0.040, 0.200, 0.456$ with $\delta=1$ and $\epsilon = 0.01$.}
\label{fig_t2_2}
\end{figure}

\begin{figure}[thbp]
\centering
\includegraphics[height=1.8in,width=2.4in]{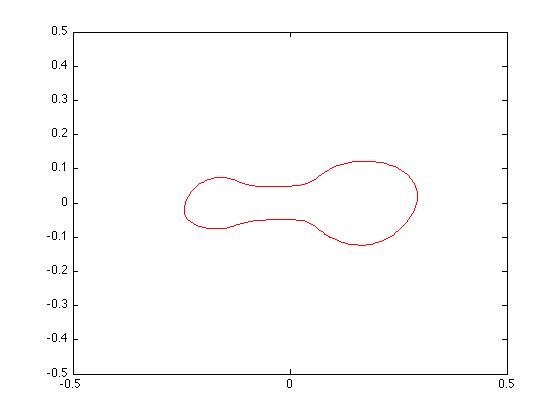}
\includegraphics[height=1.8in,width=2.4in]{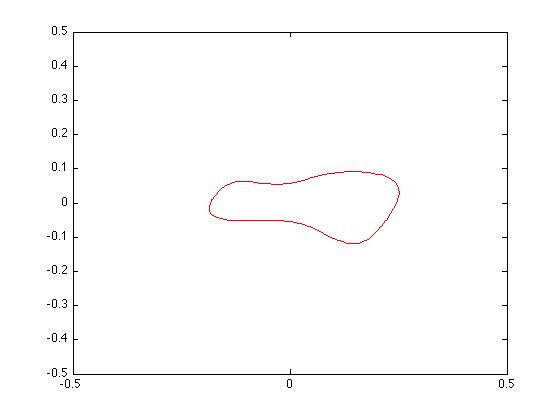}

\includegraphics[height=1.8in,width=2.4in]{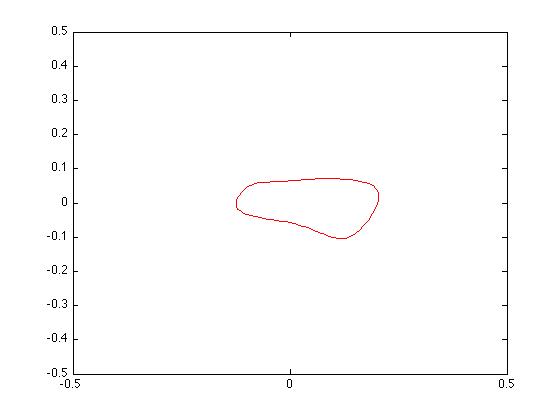}
\includegraphics[height=1.8in,width=2.4in]{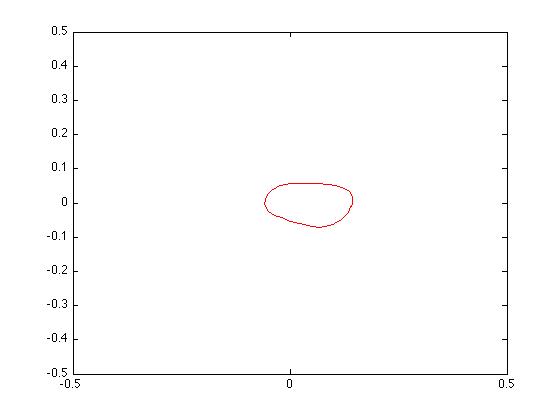}

\includegraphics[height=1.8in,width=2.4in]{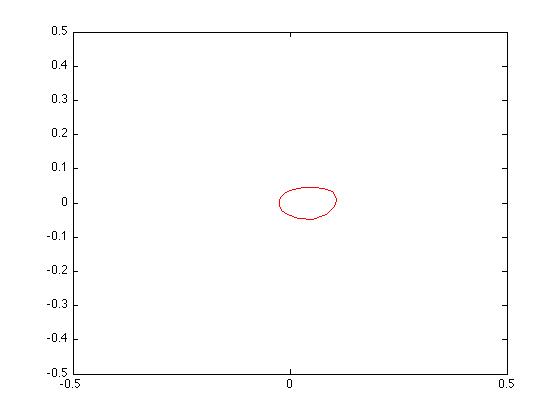}
\includegraphics[height=1.8in,width=2.4in]{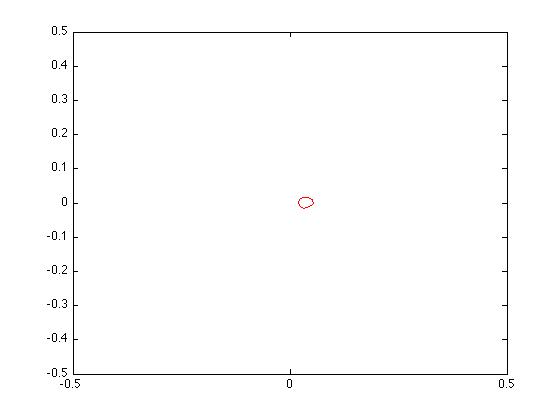}
\caption{Snapshots of the zero-level set of $\bar{u}^{\delta,\epsilon,h,\tau}$ at time 
$t=0.004, 0.040, 0.080, 0.140, 0.180, 0.216$ with $\delta=10$ and $\epsilon = 0.01$.}
\label{fig_t2_3}
\end{figure}

\begin{figure}[thbp]
\centering
\includegraphics[height=1.8in,width=2.4in]{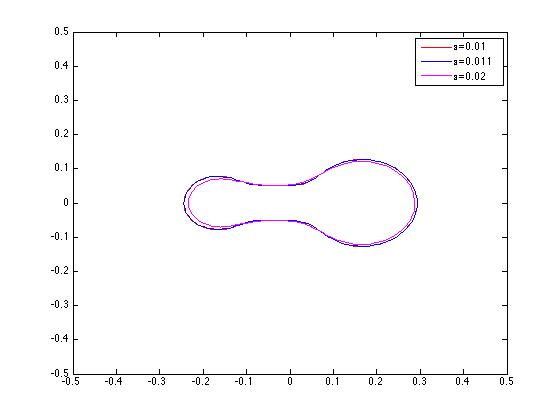}
\includegraphics[height=1.8in,width=2.4in]{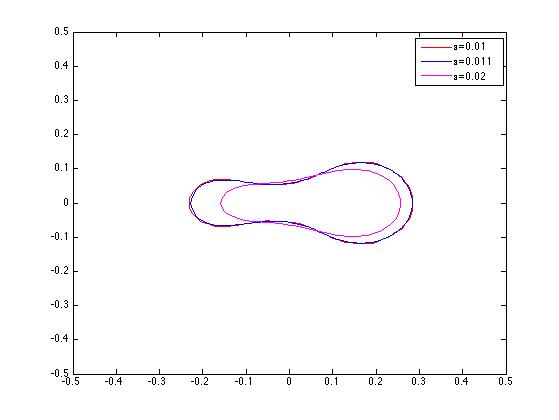}

\includegraphics[height=1.8in,width=2.4in]{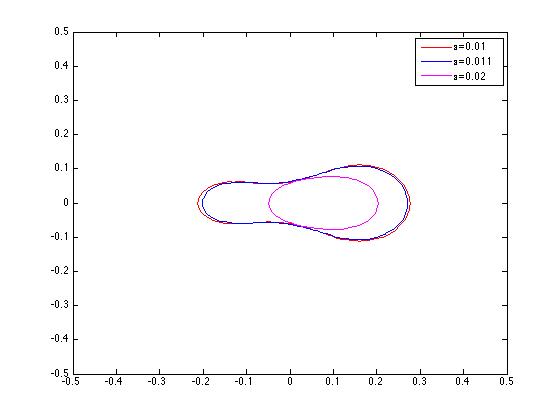}
\includegraphics[height=1.8in,width=2.4in]{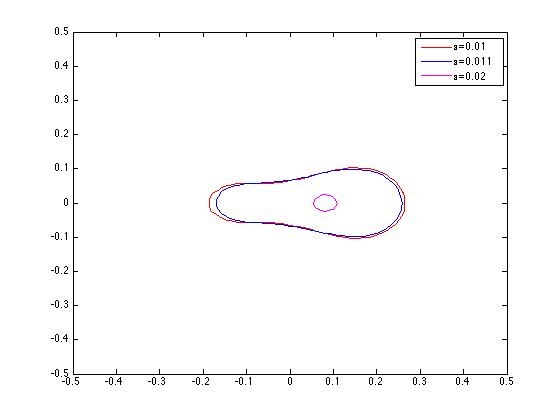}
\caption{Snapshots of the zero-level set of $\bar{u}^{\delta,\epsilon,h,\tau}$ 
at time $t=0.008, 0.040, 0.080, 0.120$.}
\label{fig_t2_4}
\end{figure}


\end{document}